\documentclass[aop,preprint]{imsart}

\RequirePackage[OT1]{fontenc}
\RequirePackage{amsthm,amsmath}
\RequirePackage[numbers]{natbib}
\RequirePackage[colorlinks,citecolor=blue,urlcolor=blue]{hyperref}
\usepackage{amssymb,mathrsfs}
\usepackage{color}
\usepackage[final]{graphicx} 
\usepackage{tikz}
\usepackage[english]{babel}

\arxiv{arXiv:1703.01617}

\startlocaldefs

\renewcommand{\r}{\mathbb{R}}

\renewcommand{\tilde}{\widetilde} 
\renewcommand{\epsilon}{\varepsilon}

\numberwithin{equation}{section}
\theoremstyle{plain}
\newtheorem{prop}{Proposition}[section]
\newtheorem{thm}[prop]{Theorem}
\newtheorem{coroll}[prop]{Corollary}
\newtheorem{lemma}[prop]{Lemma}

\theoremstyle{definition}

\newtheorem{remark}[prop]{Remark}

\newtheorem{example}[prop]{Example}
\newtheorem{assumption}[prop]{Assumption}

\newcommand{\lyp}{{\mathcal V}} 
\endlocaldefs

\begin{document}

\begin{frontmatter}
\title{Couplings and quantitative contraction rates for Langevin dynamics\thanksref{T1}}
\runtitle{Couplings for the Langevin equation}
\thankstext{T1}{Financial support from DAAD and French government through the PROCOPE program, and from the German Science foundation through the {\em Hausdorff Center for Mathematics} is gratefully acknowledged.}

\begin{aug}
\author{\fnms{Andreas} \snm{Eberle}\thanksref{m1}\ead[label=e1]{eberle@uni-bonn.de}\ead[label=u1,url]{http://wt.iam.uni-bonn.de}},
\author{\fnms{Arnaud} \snm{Guillin}\thanksref{m2}\ead[label=e2]{guillin@math.univ-bpclermont.fr}\ead[label=u2,url]{http://math.univ-bpclermont.fr/~guillin/}}
\and 
\author{\fnms{Raphael} \snm{Zimmer}\thanksref{m1}
\ead[label=e3]{raphael.zimmer@uni-bonn.de} 
\ead[label=u3,url]{http://wt.iam.uni-bonn.de}}

\runauthor{A. Eberle, A. Guillin, R. Zimmer}

\affiliation{University of Bonn\thanksmark{m1} and Universit\'e Blaise Pascal\,\thanksmark{m2}}

\address{Universit\"at Bonn\\
Institut f\"ur Angewandte Mathematik\\
 Endenicher Allee 60\\
  53115 Bonn, Germany\\
\printead{e1}\\
\phantom{E-mail:\ }\printead*{e3}\\
\printead{u1}}

\address{Laboratoire de Math\'ematiques Blaise Pascal\\
CNRS - UMR 6620\\
Universit\'e Clermont-Auvergne\\
Avenue des landais,\\
63177 Aubiere cedex, France\\
\printead{e2}\\
\printead{u2}}
\end{aug}

\begin{abstract}
\ We introduce a new
probabilistic approach to quantify convergence to equilibrium for
(kinetic) Langevin processes. In contrast to previous analytic approaches
that focus on the associated kinetic Fokker-Planck equation, our
approach is based on a specific combination of reflection
and synchronous coupling of two solutions of the Langevin equation. 
It yields contractions in a particular Wasserstein distance, and it
provides rather precise bounds for convergence to equilibrium at the
borderline between the overdamped and the underdamped regime.
In particular, we are able to recover kinetic behavior in terms
of explicit lower bounds for the contraction rate. For example, for
a rescaled double-well potential with local minima at distance $a$,
we obtain a lower bound for the 
contraction rate of order $\Omega (a^{-1})$  provided the friction coefficient is of order $\Theta (a^{-1})$.
\end{abstract}



\end{frontmatter}

\section{Introduction}
Suppose that $U$ is a function in $C^1(\mathbb R^d)$ such that
$\nabla U$ is Lipschitz continuous, and
let $u,\gamma \in (0,\infty )$. 
We consider a (kinetic) Langevin diffusion $(X_t,V_t)_{t\ge 0}$
with state space $\mathbb R^{2d}$ that is given by
the stochastic differential equation
\begin{eqnarray}
\label{Langevin}dX_t &=&V_t\, dt,\\
\nonumber dV_t&=&-\gamma V_t\, dt\, -\,  u \, \nabla U(X_t)\, dt\,
+\, \sqrt{2\gamma u}\, dB_t.
\end{eqnarray}
Here 
$(B_t)_{t\ge 0}$ is a $d$-dimensional Brownian motion that is defined on a probability
space $(\Omega ,\mathcal A,\mathbb P)$. 
Since the coefficients are Lipschitz continuous,
a unique strong solution of the Langevin equation exists for any initial
condition, and the solution gives rise to a strong Markov process with
generator
\begin{equation}\label{Gen}
\mathcal L\ =\ u\gamma\,\Delta_v\, -\, \gamma\, v\cdot\nabla_v
\, -\, u\,\nabla U(x)\cdot\nabla_v\, +\, v\cdot\nabla_x\, .
\end{equation}
The corresponding Kolmogorov forward equation is the {\em kinetic Fokker--Planck equation}. Under the assumptions on $U$ imposed below, it can be verified
that $\exp (-U)\in L^1(\mathbb R^d)$, and that the probability measure
\begin{equation}
\label{stat}
\mu_\ast (dx\, dv)\, =\, \mathcal{Z}^{-1}\,e^{ -U(x)-\frac{|v|^2}{2u}}\;
dx\, dv,\quad   \mathcal Z\, =\, (2\pi u)^{d/2}\int e^{-U(x)}\, dx,
\end{equation}  
is invariant for the transition semigroup $(p_t)_{t\ge 0}$, see e.g.\
\cite[Prop.\ 6.1]{MR3288096}.
\medskip

In statistical physics, the Langevin equation \eqref{Langevin} describes the motion of a 
particle with position $X_t$ and velocity $V_t$ in a force field $b=-\nabla U$ 
subject to damping and random collisions \cite{einstein1905molekularkinetischen,
von1906kinetischen, nelson1967dynamical,langevin1908theorie,MR2583642}. In the physical
interpretation, $\gamma$ is the friction coefficient (per unit mass), and $ u $ is the inverse mass. Discretizations of the Langevin equation are relevant for molecular dynamics simulations \cite{MR2681239}.
Hamiltonian Monte Carlo methods for sampling and
integral estimation are based on different types of discrete
time analogues to Langevin dynamics \cite{duane1987hybrid,neal,MR2681239,2015arXiv151109382B}.
In numerical simulations, often a better performance of these HMC methods
compared to traditional MCMC approaches is observed, but the
corresponding convergence acceleration is still not well understood theoretically.
\medskip

For these and other reasons, an important question is how to obtain explicit bounds on the
speed of convergence of the law of $(X_t,V_t)$ towards the invariant probability measure $\mu_\ast$. Since the noise is only acting on the second 
component, the generator of the Langevin diffusion is degenerate, and thus classical approaches can
not be applied in a straightforward way. Indeed, $\mathcal L$ is a 
typical example of a hypocoercive operator in the sense of Villani
\cite{MR2339441,MR2562709}. Several analytic
approaches to convergence to equilibrium for kinetic
Fokker-Planck equations have been proposed during the last 15 years \cite{MR1787105,MR1969727,MR2034753,MR2130405,
MR2215889,MR2339441,MR2562709,MR3324910,
MR2957543,MR3488535,GrothausStilgenbauerJFA, GrothausStilgenbauerFAT,  2016arXiv160204177B}.
These are based respectively on Witten Laplacians and functional inequalities, semigroup theory, 
and in particular on hypocoercivity methods,
see also \cite{GM16} for some explorations around the Gaussian
case and the effect of hypoellipticity.
There are only few
articles which study the ergodic properties of Langevin processes using more probabilistic arguments, 
cf.\ \cite{MR1924934,MR1807683,MR1931266,MR1889227,
BCG08,MR2731396}. Most of these results ultimately rely
on arguments used in Harris' type theorems, i.e., they assume a Lyapunov
drift condition which implies recurrence of the process w.r.t.\ a
compact set together with a control over the average excursion
length.
This condition is then combined with an argument showing that for
starting points in the recurrent set,
the transition probabilities 
are not singular w.r.t.\ each other. While the approaches are of a probabilistic nature, the behaviour
of the process inside the recurrent set is not very transparent. Correspondingly, these approaches lead 
to qualitative rather than quantitative convergence results.
\smallskip

An open question asked by Villani in
\cite[Ch.\ 2, Bibliographical notes]{MR2459454} is how to prove 
exponential convergence to equilibrium by a direct {\em coupling approach}. 
The motivation for this is two-fold: On the one hand, coupling methods often
provide a good probabilistic understanding of the dynamics. On the other hand,
couplings have been proven useful in establishing precise bounds on the long-time behaviour of non-degenerate diffusion processes
\cite{LindvallRogers, CL, EberlePTRF, 2016arXivE}. The only results for Langevin processes in this direction 
that we are aware of are rather restrictive:
Under the assumption that the force field $\nabla U$
is a small perturbation of a linear function, Bolley, Guillin and Malrieu 
\cite{MR2731396} use a synchronous
coupling to show exponential mixing for \eqref{Langevin} in $L^2$ Wasserstein distances.  Moreover, in \cite{BenArousCranstonKendall,BanerjeeKendall}, 
couplings for the Kolmogorov diffusion have been considered.
This process solves an equation similar to \eqref{Langevin} without
damping and with $U\equiv 0$.
\smallskip

Here, 
we develop a novel coupling approach for Langevin equations that works for a much wider class of force fields. We briefly describe the main ideas behind this approach: A coupling of two solutions of \eqref{Langevin}
is given by stochastic processes $(X_t,V_t)_{t\ge 0}$ and $(X'_t,V'_t)_{t\ge 0}$
with state space $\mathbb R^{2d}$ that are defined on a common probability space and
satisfy \eqref{Langevin} and, respectively,
\begin{eqnarray}
\label{Langevincopy}dX'_t &=&V'_t\, dt,\\
\nonumber dV'_t&=&-\gamma V'_t\, dt\, -\,  u \, \nabla U(X'_t)\, dt\,
+\, \sqrt{2\gamma u}\, dB'_t,
\end{eqnarray}
where $(B_t)_{t\ge 0}$ and $(B'_t)_{t\ge 0}$ are $d$-dimensional Brownian motions. The only 
freedom in constructing a coupling is the way these Brownian motions are related 
to each other. For a synchronous coupling, $B_t=B'_t$ for all $t$. In this case, 
the difference process $(Z_t,W_t)=(X_t-X'_t,V_t-V'_t)$ satisfies a deterministic
o.d.e., and contractivity holds if and only if it holds for the equation without noise.
This applies for example for overdamped Langevin diffusions in a strictly convex 
potential or on a positively curved Riemannian manifold, but in general it is a  rather restrictive condition that is not satisfied in our case. Nevertheless, 
one can observe that w.r.t.\ an appropriately chosen metric on $\mathbb R^{2d}$, the difference process is contractive without noise as long as it is in a
neighbourhood of the hyperplane where $Q_t:=Z_t+\gamma^{-1 }W_t=0$, see
Section \ref{sub:coupling} below. Therefore, synchronous coupling can be applied 
in this region. 

If the dynamics is not contractive, one has to exploit the random fluctuations to
ensure that the two copies approach each other in some sense. A well-known
approach is reflection coupling \cite{LindvallRogers} where the noise increments
$dB_t$ and $dB'_t$ are synchronized in directions orthogonal to the difference
of the two copies and reflected in the direction connecting the copies. As a 
consequence, the difference process is driven by a one-dimensional noise in this
direction. It has been shown in \cite{CW,MR2843007,EberlePTRF,2016arXivE} that this can be exploited to obtain average contractivity with relatively sharp explicit rates
in distances that are appropriately chosen concave functions of $\ell^1$ or $\ell^2$ metrics. This approach works well for non-degenerate diffusions but it
fails for the degenerate case. Therefore it does not apply directly to the Langevin 
equation. Nevertheless, it can be used in the directions complementary to the
contractive hyperplane. 

Combining the two types of couplings above suggests that we should apply a 
coupling that is synchronous whenever $Q_t$ equals $0$ (or is close to $0$), and a reflection coupling in the complementary directions otherwise. This means we should set
\begin{equation}\label{stickycoup}
dB'_t\ =\ (I_d- 1_{\{ Q_t\neq 0\} }2e_te_t^T)   \, dB_t
\end{equation}
where $e_t=Q_t/|Q_t|$.  Then the resulting coupling difference process will be driven by
noise whenever $Q_t\neq 0$, and the noise will be switched off if $Q_t=0$. 
L\'evy's characterization ensures that $(B'_t)$ is again a Brownian motion, and the
resulting coupling process is a diffusion process on $\mathbb R^{4d}$ that is
\emph{sticky} \cite{Watanabe1,Watanabe2,sticky} on the subspace $\{ (x,v,x',v')\in\mathbb R^{4d}:x-x'+\gamma^{-1}(v-v')=0\}$
where contractivity holds without noise. This means that almost surely, after reaching the subspace due to reflection coupling, the process spends 
a positive amount of time on this subspace although it does not stay on the
subspace for any positive time-interval. Each time it leaves the subspace, it immediately returns due to the random fluctuations that are switched on when 
$Q_t\neq 0$. In total, the set $\{ t\in [0,\infty ): Q_t=0\}$ of all times where the process visits the subspace has almost surely positive Lebesgue measure although it does 
not contain any non-empty open interval.
The rigorous construction of a corresponding 
sticky coupling can be carried out by a weak convergence approach that is based on approximating the discontinuous coefficients in \eqref{stickycoup} by Lipschitz
continuous functions. This has been done 
in a slightly different setup in
\cite{sticky}. In general, the stochastic differential equation for the corresponding
sticky coupling process does not have a strong solution but the approximation 
procedure yields a weak solution, see \cite{sticky} for details. 

Since the
construction and control of the sticky coupling described above is possible
but technically involved, we actually do not consider the sticky coupling itself here. Instead, we 
use approximations of such a coupling in order to derive bounds for contraction rates, see \eqref{coupling} below. The corresponding limit is taken only in 
the resulting bounds, and the construction of the sticky coupling itself (i.e., the limit of the approximating coupling processes) is not required for our results.
The speed of convergence is then measured in Kantorovich distances ($L^1$ Wasserstein distances) by adapting and optimizing the underlying (semi-)metric
w.r.t.\ the given model and the chosen coupling. Here, we basically follow the strategy developed in \cite{2016arXivE} which extends the results in 
\cite{MR2843007,EberlePTRF}. The approach taken in \cite{2016arXivE}, which is partially based on ideas from \cite{MR2857021,MR2773030, CW}, is to build a multiplicative
semi-metric $\rho$ out of a concave function of the underlying distance and a 
Lyapunov function that ensures contractivity at large distances, see
Section \ref{sub:metric} below. In a slight modification of \eqref{stickycoup}, we will apply 
synchronous coupling at large distances, since here a Lyapunov drift condition 
will ensure contractivity.
Both the concave functions and the constants
entering the definition of the metric $\rho$ (see \eqref{rho} and \eqref{rDist}) are carefully
chosen in order to optimize the order of the resulting contraction rates.\medskip

The approach for studying long-time stability properties of diffusion processes
by using sticky couplings (or corresponding approximations) seems to be useful in many different contexts, see 
\cite{sticky}. For example, it is also related
to the application of a similar strategy
to infinite-dimensional stochastic differential equations with possibly
degenerate noise in \cite{zimmer16}, cf.\ also \cite{MR1939651,MR1937652}.
In general, the idea is to identify some submanifold of the state space for the coupling where contraction properties hold for the equation without noise. Then 
synchronous coupling can be applied on this submanifold whereas outside,
random fluctuations introduced by a different coupling ensure that the
process reaches the submanifold in finite time.\smallskip

Besides providing an intuitive understanding for the mechanism
of convergence to equilibrium, the coupling approach yields both qualitatively
new, and explicit quantitative results in several cases of interest.
Before explaining the coupling construction and stating the results in detail, we illustrate this by
an example:

\begin{example}[Double-well potential]
Suppose that $U\in C^1(\mathbb R )$ is a Lipschitz continuous double-well potential defined by
$$U(x)\ =\ 
\begin{cases} (|x|-1)^2/2 &\mbox{for }|x|\ge 1/2,\\ 1/4-|x|^2/2&\mbox{for }|x|\le 1/2,\end{cases}  $$
and let $U_a(x)=U(x/a)$ be the rescaled potential
with the same height for the potential well, but minima at distance $2 a$.
Then our main result shows that for any $a,u,\gamma\in (0,\infty )$ there exist a constant $c\in (0,\infty )$, a semimetric $\rho$ on $\mathbb R^{2d}$, and a corresponding 
Kantorovich semimetric $\mathcal W_\rho$ such that for all 
probability measures $\mu ,\nu$ on $\mathbb R^{2d}$,
\begin{equation*}
\mathcal W_\rho (\mu p_t,\nu p_t)\ \le \ e^{-ct}\,\mathcal W_\rho
(\mu ,\nu )\qquad\mbox{for any }t\ge 0.
\end{equation*}
As a consequence, we also obtain convergence to equilibrium in the
standard $L^2$ Wasserstein distance with the same exponential rate $c$.
Below, we give explicit lower bounds for the contraction rate. For example, if $\gamma a\ge \sqrt{30\, u}$ then 
$$c\ \ge\ \frac{\sqrt u}{107}\,\min\left( (\gamma a)^{-4}u^2, \, e^{-8}(\gamma a)^{-2}u,\, 2^{-3/2}e^{-8}\right) \, a^{-1}, $$
see Example \ref{ex:multiwell} (with parameters $\mathcal R=4a$, $L=a^{-2}$ and $\beta =L\mathcal R^2/2=8$). In general, if the value of 
$\gamma a$ and $u$ are fixed (i.e., the friction coefficient $\gamma$ is adjusted to the potential) then the contraction rate is of order $\Omega (a^{-1})$, i.e., $c\ge c_0\cdot a^{-1}$ for a positive constant $c_0$. This
clearly reflects the {\em kinetic behaviour}, and it is in contrast to the rate $O(a^{-2})$ for 
convergence to equilibrium of the overdamped limit $dX_t=-u \,\nabla U(X_t)\, dt+\sqrt{2u}\, dB_t$. 
\end{example}

\section{Main results}

\subsection{Coupling construction}\label{sub:coupling}

We first explain the
construction of the coupling briefly, see Section \ref{sec:coupling} for full details. Suppose
that $((X_t,V_t),(X'_t,V'_t))$ is an arbitrary coupling of two solutions of
the Langevin equation \eqref{Langevin} driven by Brownian motions 
$(B_t)$ and $(B'_t)$. Then the difference process $(Z_t,W_t)=(X_t-X'_t,
V_t-V'_t)$ satisfies the stochastic differential equation
\begin{eqnarray*}
dZ_t &=&W_t\, dt,\\
\nonumber dW_t&=&-\gamma W_t\, dt\, -\,  u \, (\nabla U(X_t)-\nabla U(X'_t))\, dt\,
+\, \sqrt{2\gamma u}\, d(B-B')_t.
\end{eqnarray*}
Introducing the new coordinates $Q_t=Z_t+\gamma^{-1}W_t$, the system 
takes the form
\begin{eqnarray}\label{eq:AA}
dZ_t &=&-\gamma Z_t\, dt\, +\, \gamma Q_t\, dt,\\
dQ_t&=& -\,  u\gamma^{-1} \, (\nabla U(X_t)-\nabla U(X'_t))\, dt\,
+\, \sqrt{2 u\gamma^{-1}}\, d(B-B')_t.\label{eq:AB}
\end{eqnarray}
Since $\gamma >0$, the first equation is contractive if $Q_t=0$. The key
idea is now to apply a synchronous coupling whenever $Q_t=0$, and a
reflection coupling if $Q_t\neq 0$ and $\alpha |Z_t|+|Q_t|<R_1$ with
appropriate constants $\alpha, R_1\in (0,\infty )$, cf.\ Figure
\ref{figCoupling}. The synchronous coupling guarantees that the noise
coefficient in \eqref{eq:AB} vanishes if $Q_t=0$, i.e., the dynamics is not driven away from the ``contractive region'' by
random fluctuations (although it may leave this region by the drift).
On the other hand, the reflection coupling for $Q_t\neq 0$ ensures that
the contractive region is recurrent. The resulting coupling process is a
diffusion on $\mathbb R^{4d}$ that is sticky on the $3d$-dimensional
hyperplane $\{ (x,v,x',v')\in\mathbb R^{4d}:x-x'+\gamma^{-1}(v-v')=0\}$
of contractive states, i.e., it spends a positive amount of
time in this region, cf.\ \cite{sticky}. Since the
construction and control of the sticky couplings described above is possible
but technically involved, we actually use approximations of such couplings to derive our results, see \eqref{coupling}. 
By designing a special semi-metric $\rho$ on $\mathbb R^{2d}$ that is based on a concave function of the
distance and on a Lyapunov function, we can then (similarly as in \cite{2016arXivE})  make use of the random fluctuations and of a drift
condition in order to derive average contractivity for the coupling distance
$\rho ((X_t,V_t),(X'_t,V'_t))$.
The construction of a coupling and the proof of contractivity are carried
out rigorously in Sections \ref{sec:coupling} and \ref{sec:contrregion}. 
 
\begin{center}
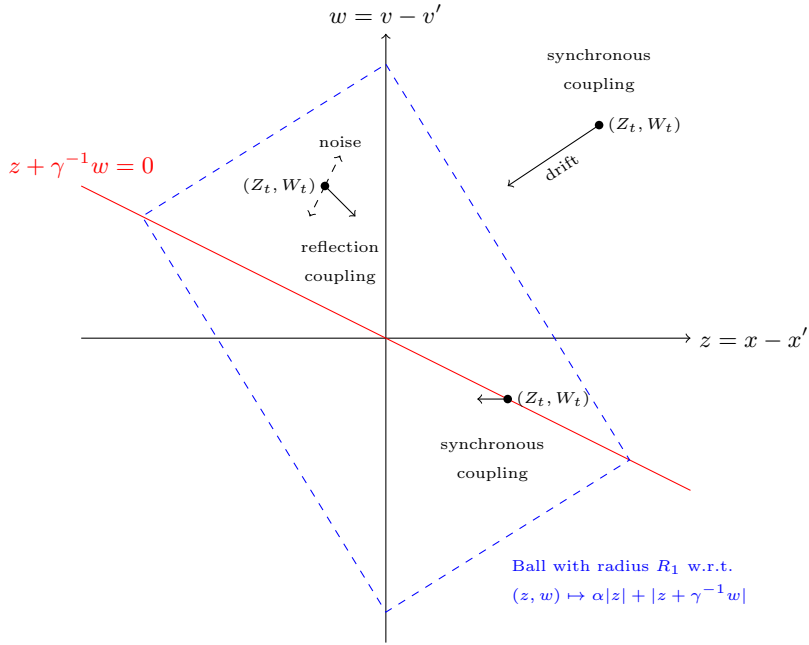
\begin{figure}
\begin{tikzpicture}[scale=0.81] 
\draw[->] (-5,0) coordinate (xAxisNeg) -- (5,0) node (xAxisPos)  [right] {$z=x-x'$};
\draw[->] (0,-5) coordinate (yAxisNeg) -- (0,5) node (yAxisPos) [above] {$w=v-v'$};
\draw[-,  red] (-5,2.5) node (ContrLine) [above] {$z+\gamma^{-1} w = 0$} -- (5,-2.5);
\draw[-,  blue, dashed=blue] (0,4.5) -- (4,-2) -- (0,-4.5) --(-4,2) -- (0,4.5);
\node[align=left] at (4,-4) {\color{blue}\tiny{Ball with radius $R_1$
w.r.t.}\\
\color{blue}\tiny{$(z,w)\mapsto \alpha |z|+|z+\gamma^{-1} w|$}};

\coordinate [label={right: \tiny{$(Z_t,W_t)$}}] (A1) at (3.5,3.5);
\fill[black] (A1) circle (2pt) ;
\coordinate (A2) at (3.5,4.4);
\node[align=center] at (A2)  {\tiny{synchronous}\\\tiny{coupling}};
\coordinate (A3) at (2,2.5);
\path (A1) edge[draw=black,->]
node[sloped,anchor=center,below] {\tiny{drift}} (A3); 

\coordinate [label={left: \tiny{$(Z_t,W_t)$}}] (B1) at (-1,2.5);
\fill[black] (B1) circle (2pt) ;
\coordinate [label={above: \tiny{noise}}] (B2) at (-0.75,3);
\path (B1) edge[draw=black,->,dashed] (B2);
\draw[->, dashed] (B1)  -- (-1.25,2);
\node[align=center] at (-0.75,1.25)  {\tiny{reflection}\\\tiny{coupling}};
\coordinate (B3) at (-0.5,2);
\path (B1) edge[draw=black,->] (B3); 

\coordinate [label={right: \tiny{$(Z_t,W_t)$}}] (C1) at (2,-1);
\fill[black] (C1) circle (2pt) ;
\coordinate (C2) at (1.75,-2);
\node[align=center] at (C2)  {\tiny{synchronous}\\\tiny{coupling}}; \draw[->] (C1)  -- (1.5,-1);

\end{tikzpicture}
\caption{Sketch of coupling approach\label{figCoupling}}
\end{figure}
\end{center}

\subsection{Drift condition and Lyapunov function}

We now make the following assumption that guarantees, among other things, that the process is non-explosive:
\begin{assumption}\label{ass:1}
There exist constants $L,A\in (0,\infty )$ and $\lambda\in (0,1/4]$ such that
\begin{align}
\label{A0a}U(x)\ &\ge \ 0 \qquad\mbox{for all }x\in\mathbb R^d,\\
\label{A1a}|\nabla U(x)-\nabla U(y)| \ &\le \ L\, |x-y|\qquad
\mbox{for all }x,y\in\mathbb R^d,\quad\mbox{and}\\
\label{A2}
x\cdot\nabla U(x)/2 \ &\ge \ \lambda\, (U(x)+ u^{-1}\gamma^2|x|^2/4)-A\quad\mbox{for all }x\in\mathbb R^d.
\end{align}
\end{assumption}
Notice that the assumption
can only be satisfied if
\begin{equation}
\label{lambda1}
\lambda\ \le\ 2Lu\gamma^{-2}.
\end{equation}
Up to the choice of the constants, the drift condition \eqref{A2} is equivalent to the 
simplified drift condition \eqref{A2prime} considered further below. It
implies the existence of a Lyapunov function for the Langevin process. Indeed, let
\begin{equation}
\label{H}
\lyp(x,v)\ =\ U(x)\, +\, \frac 14u^{-1}\gamma^2\, \left(
|x+\gamma^{-1}v|^2+|\gamma^{-1}v|^2-\lambda |x|^2\right) . 
\end{equation}
Note that since $\lambda\le 1/4$,
\begin{eqnarray}
\nonumber \lyp(x,v) &\ge & U(x)\, +\, \frac 14 (1-2\lambda )u^{-1}\gamma^2\left(
|x+\gamma^{-1}v|^2+|\gamma^{-1}v|^2\right)\\
\label{H1}&\ge &\frac 18 (1-2\lambda )u^{-1}\gamma^2 |x|^2.
\end{eqnarray}
In particular, $\lyp(x,v)\to\infty$ as $|(x,v)|\to\infty $.
Moreover:

\begin{lemma}\label{lem:L}
If the drift condition \eqref{A2} holds then $\, \mathcal L\lyp\,\le\,\gamma\, (d+A-\lambda \lyp)$.
\end{lemma}

The proof of the lemma is included in Appendix \ref{app:Lyapunov}. The choice of the
Lyapunov function is motivated by Mattingly, Stuart and Higham
\cite{MR1931266}, see also \cite{MR1924934,MR1807683,BCG08}.
In combination
with \eqref{H1}, the lemma shows that the process $\lyp(X_t,V_t)$ is decreasing on average in regions where $$|X_t|\ \ge\ 8^{1/2}(d+A)^{1/2}u^{1/2}\gamma^{-1}(\lambda
-2\lambda^2)^{-1/2}.$$

\subsection{Choice of metric}\label{sub:metric}

Next, we introduce an appropriate semi-metric on $\mathbb R^d$ w.r.t.\
which the coupling considered below will be contractive on average. Inspired by \cite{MR2857021}, a similar semi-metric
has been considered in \cite{2016arXivE}.
For $(x,v),(x',v')\in\mathbb R^{2d}$ we set
\begin{eqnarray}
\label{rDist}\qquad \ r((x,v),(x',v')) &=& \alpha |x-x'|+|x-x'+{\gamma}^{-1}(v-v')|\\
\label{rho}\qquad \ \rho ((x,v),(x',v')) &=& f(r((x,v),(x',v')))\cdot (1+ \epsilon
\lyp(x,v)+\epsilon \lyp(x',v')), 
\end{eqnarray}
where $\alpha ,\epsilon\in (0,\infty )$ are appropriately chosen
positive constants, and $f:[0,\infty )\to [0,\infty )$ is a continuous,
non-decreasing concave function such that $f(0)=0$, $f$ is $C^2$
on $(0,R_1)$ for some constant $R_1\in (0,\infty )$ with 
right-sided derivative $f'_+(0)=1$ and left-sided derivative
$f'_-(R_1)>0$, and $f$ is constant on $[R_1,\infty )$. The 
function $f$ and the constants $\alpha ,\epsilon $ and $R_1$ will
be chosen explicitly below in order to optimize the resulting 
contraction rates. For the moment let us just note that by concavity,
\begin{equation}
\label{boundf}
\min (r,R_1)\, f'_{-}(R_1)\le f(r) \le\min (r,f(R_1)) \le
\min (r,R_1) \quad\mbox{for }r\ge 0.
\end{equation}
For probability measures $\mu ,\nu$ on $\mathbb R^{2d}$ we define
\begin{equation}
\label{0b}
\mathcal W_\rho (\mu ,\nu )\ =\ \inf_{ \Gamma \in \Pi (\mu ,\nu )}
\int \rho ((x,v),(x',v'))\; \Gamma  (d(x,v)\, d(x',v'))
\end{equation} 
where the infimum is over all couplings of $\mu $ and $\nu$. We remark
that $\rho$ and the transportation cost $\mathcal W_\rho$
are semimetrics but not necessarily metrics, i.e., the triangle inequality may be violated. An important remark is that the distance $r$ can be controlled by the Lyapunov function. Indeed, let
\begin{equation}
\label{R1}R_1\ :=\ \left( 16\cdot\left( 6/5\right)\cdot \left( 1+2\alpha
+2\alpha^2\right)\left( d+A\right)u\gamma^{-2}\, 
\left( \lambda -2\lambda^2\right)^{-1}\right)^{1/2}. 
\end{equation}
By \eqref{H1} and since $U\ge 0$,
\begin{align}\label{ML}
\lefteqn{r((x,v),(x',v'))^2\ \le\ \left( (1+\alpha )|x-x'+\gamma^{-1}(v-v')|
+\alpha |\gamma^{-1} (v-v')|\right)^{2}}\\\nonumber
\qquad \ &\le \ 2\, ((1+\alpha )^2+\alpha^2)\left(
|x+\gamma^{-1}v|^2+|x'+\gamma^{-1}v'|^2+|\gamma^{-1}v|^2+|\gamma^{-1}v'|^2\right)\\
\nonumber \ &\le \  8\, ((1+\alpha )^2+\alpha^2)(1-2\lambda )^{-1}
u\gamma^{-2}\, \left( \lyp(x,v)+\lyp(x',v')
\right)
\end{align}
for any $(x,v),(x',v')\in\mathbb R^{2d}$. Hence for 
$r((x,v),(x',v'))\ge R_1$,
\begin{eqnarray}
\label{H2}\lyp(x,v)+\lyp(x',v') &\ge &\frac{12}{5	}(d+A)/\lambda ,
\qquad \mbox{and thus}\\
\label{Lyap}\mathcal L \lyp(x,v)+\mathcal L \lyp(x',v') &\le &
-\frac 16\, \gamma\lambda\, \left( \lyp(x,v)+\lyp(x',v')\right)
\end{eqnarray}
by Lemma \ref{lem:L}. The bound \eqref{Lyap} guarantees that for
the coupling to be considered below, the process
$\rho_t:=\rho ((X_t,V_t),(X'_t,V'_t))$ is decreasing on average if
$r_t:=r((X_t,V_t),(X'_t,V'_t))\ge R_1$. We show that by choosing the coupling and the parameters $\alpha ,\epsilon$ and $f$
defining the metric in an adequate way, we can ensure that $\rho_t$
is also decreasing on average (up to a small error term) for
$r_t<R_1$. As a consequence, we will obtain our basic 
contraction result.

\subsection{Main contraction result}

We can now state our main result:
\begin{thm}\label{thm:1}
Suppose that Assumption \ref{ass:1} is satisfied.
Then there exist constants $\alpha, \epsilon\in (0,\infty )$ and a 
continuous non-decreasing concave function $f:[0,\infty )\to
[0,\infty )$ with $f(0)=0$ such that for all probability measures $\mu ,\nu$ on $\mathbb R^{2d}$,
\begin{equation}
\label{contr}
\mathcal W_\rho (\mu p_t,\nu p_t)\ \le \ e^{-ct}\,\mathcal W_\rho
(\mu ,\nu )\qquad\mbox{for any }t\ge 0,
\end{equation}
where the contraction rate $c$ is given by
\begin{align}
\label{rate}
c\ &=\  \frac{\gamma}{384}\,\min\left( \lambda L u \gamma^{-2},\,
\Lambda^{1/2}e^{-\Lambda}Lu\gamma^{-2},\, \Lambda^{1/2}e^{-\Lambda}\right)\qquad\mbox{with}\\
\label{Lambda}\Lambda \ &:= \ LR_1^2/8\ =\ \frac{12}{5}
(1+2\alpha +2\alpha^2 )(d+A)Lu\gamma^{-2}\lambda^{-1}(1-2\lambda )^{-1}.
\end{align}
Explicitly, one can choose
 the constants $\alpha ,\epsilon $, and
the function $f$ determining $\rho$ in such a way that
\begin{equation}
\label{alphaeps}\alpha\ =\ (1+\Lambda^{-1})Lu\gamma^{-2}\ \le \ 
\frac{11}{6	}Lu\gamma^{-2},\qquad \epsilon\ =\ 4\gamma^{-1}c/(d+A),
\end{equation}
and $f$ is constant on $[R_1,\infty )$ and $C^2$ on $(0,R_1)$ with
\begin{align}\label{fboundsTheorem1}
\frac 12 e^{-2}\exp (-Lr^2/8)\ \le\ f'(r)\ \le\ \exp (-Lr^2/8)\quad
\mbox{for }r\in (0,R_1).
\end{align}
More precisely, $f$ is defined by \eqref{f1}, \eqref{phi1} and
\eqref{g1} below.
\end{thm}

\begin{remark}
The constant $\Lambda$ depends on the parameters $L,u$ and
$\gamma$ both explicitly and through $\lambda $ and $\alpha$.
By \eqref{lambda1}, we always have
\begin{equation}
\label{Lambda1}\Lambda\ \ge \ 6(d+A)/5\ \ge\ 6/5.
\end{equation}
Corresponding upper bounds are given in Lemma \ref{lem:sdc}
below.
\end{remark}
\begin{remark}
We shortly comment on the requirement that $\nabla U$ is Lipschitz,
cf.\ Assumption \ref{ass:1} further above. This condition is not
necessary to conclude exponential convergence to equilibrium in Kantorovich
distances, cf.\ \cite[Theorem 3.2]{MR1931266}. We have chosen to limit
ourselves here to the Lipschitz case to concentrate on the key techniques rather than on
tedious calculations. In the case of overdamped Langevin
equations, the contraction results from \cite{2016arXivE} are extended in
\cite{zimmerPhD} replacing global Lipschitz bounds by local ones. In a similar
spirit, it might be possible to extend the results presented here. However,
optimizing and keeping track of the constants is more involved in this case.
\end{remark}

The proof of Theorem \ref{thm:1} is given in Section \ref{sec:optimize}. As a preparation, we introduce the
relevant couplings in Section \ref{sec:coupling}, and we apply these
to derive a more general contraction result in Section \ref{sec:contrregion}. Theorem \ref{thm:1} will be obtained from
this more general result by choosing the constants $\alpha$
and $\epsilon$ in a specific way.\medskip

Theorem \ref{thm:1} directly implies convergence of the Langevin process to a unique stationary distribution with exponential rate $c$:

\begin{coroll}\label{cor:1A}
In the setting of Theorem \ref{thm:1} there exists a constant 
$C\in(0,\infty)$ such that for all probability measures $\mu ,\nu$ on $\mathbb R^{2d}$,
\begin{equation}
\label{AB}\mathcal W^{2}(\mu p_t,\nu p_t)^2\ \le\
Ce^{-ct}\, \mathcal W_\rho (\mu ,\nu ) \qquad\text{for any } t\ge 0.
\end{equation}
Here, $\mathcal W^{2}$ denotes the standard $L^2$ Wasserstein distance w.r.t.\
the euclidean metric.
In particular, $\mu_\ast$ is the unique invariant probability measure for the Langevin
process, and $\mu p_t$ converges towards $\mu_\ast$ exponentially
fast with rate $c$ for any initial law $\mu$ such that $\mathcal W_\rho (\mu,\mu_\ast )<\infty$. Here, the constant $c$ and the semimetric
$\rho$ are given as in Theorem \ref{thm:1}, and the constant $C$ can be chosen explicitly as
$$
	C\ =\  2\, e^{2+\Lambda}\,  \frac{(1+\gamma)^2}{\min(1,\alpha)^{2} }\, \max\left(1, \,{4\, (1+2\alpha +2\alpha^2) \frac{(d+A)
u \gamma^{-1}c^{-1}}{ \min(1,R_1)}}\right)\, 
	  .
$$
\end{coroll}

The proof is given in
Section \ref{sec:optimize}.

\subsection{Bounds under simplified drift condition}

In order to make
the dependence of the bounds on the parameters more explicit, we now replace 
\eqref{A2} by a simplified drift condition. Instead of Assumption \ref{ass:1},
we assume:
\begin{assumption}\label{ass:2}
There exist constants $L,\mathcal R,\beta\in (0,\infty )$ such that
\begin{align}
\label{A0}U(0)\ =\ 0 \ &=\ \min U,\\
\label{A1}|\nabla U(x)-\nabla U(y)| \ &\le \ L\, |x-y|\qquad
\mbox{for any }x,y\in\mathbb R^d,\quad\mbox{and}\\
\label{A2prime}x\cdot\nabla U(x) \ &\ge \  \beta \cdot (|x|/\mathcal R)^2\quad\mbox{for any }x\in\mathbb R^d\mbox{ s.t.\ }|x|\ge\mathcal R.
\end{align}
\end{assumption}

Observe that \eqref{A0} may be assumed w.l.o.g.\ by subtracting a
constant and shifting the coordinate system such that the global
minimum of $U$ (which exists if \eqref{A2prime} holds) is attained at
$0$. The Lipschitz condition \eqref{A1} has been assumed before, 
and up to the values of the constants, the drift condition \eqref{A2prime} is
equivalent to \eqref{A2}. This condition guarantees that the $x$ marginal of the invariant probability measure \eqref{stat} concentrates on balls of radius
$O(\mathcal R)$. Notice that if \eqref{A1} and
\eqref{A2prime} are both satisfied then 
\begin{equation}
\label{Lbeta}\beta\ \le\ L\mathcal R^2.
\end{equation}


\begin{lemma}
\label{lem:sdc}
Suppose that Assumption \ref{ass:2} is satisfied.
Then Assumption \ref{ass:1} holds with
\begin{align}
A\ &=\  (L\mathcal R^2-\beta )/8\quad\mbox{and}\quad
\label{lambda3}\lambda \ =\ \min\left( \frac{1}{4},\, \frac{\beta}{L\mathcal R^2	}\cdot\frac{2Lu\gamma^{-2}}{1+2Lu\gamma^{-2}}\right) .
\end{align}
Furthermore, if $Lu\gamma^{-2}\le 1/8$ then the constant $\Lambda$ in Theorem \ref{thm:1} is bounded by 
\begin{equation}
\label{Lambda2}\frac 65 (d+A)L\mathcal R^2/\beta
\ \le\ \Lambda\ \le\ \frac 65 (d+A)(1+20\, Lu\gamma^{-2})\, L\mathcal R^2/\beta.
\end{equation}
In general, there is an explicit constant $C_1\in (0,\infty )$ such that
\begin{equation}
\label{Lambda3}\frac 65 (d+A)L\mathcal R^2/\beta
\ \le\ \Lambda\ \le\ \frac {12}5 (d+A)(1+C_1 Lu\gamma^{-2})^3\, L\mathcal R^2/\beta.
\end{equation}
\end{lemma}

The proof is included in Appendix \ref{app:Lyapunov}.
The lemma shows that if Assumption \ref{ass:2} holds with fixed constants $L,\mathcal R,\beta\in (0,\infty )$,
then there is a lower bound for the contraction rate in Theorem
\ref{thm:1} that only depends on the natural parameters $\gamma$,
$Lu\gamma^{-2}$, $L\mathcal R^2$ and $\beta$. The bound is
particularly nice if there is sufficient damping:

\begin{coroll}\label{cor:sdc}
Let $\ell\in [1,\infty )$, and suppose that Assumption \ref{ass:2} is satisfied with constants
$L,\mathcal R,\beta\in (0,\infty )$ such that $\beta\ge L\mathcal R^2/\ell$. Suppose further that $Lu\gamma^{-2}\le 1/30$. Then the assertion of Theorem \ref{thm:1} holds with a contraction rate 
\begin{align}
\nonumber
\ c \ &\ge \ \frac{\gamma }{205}\, \min\left(
\frac{1}{\ell }(Lu\gamma^{-2})^2,\, \frac 12\min\left( d^{1/2}Lu\gamma^{-2},\Lambda_1^{-1/2}\right)\, e^{-\Lambda_1}\right) \\
\ &\ge \ \frac{\sqrt{\beta u}}{38}\, \min\left(
\frac{1}{\ell }(Lu\gamma^{-2})^2,\, \frac 12\min\left( d^{1/2}Lu\gamma^{-2},\Lambda_1^{-1/2}\right)\, e^{-\Lambda_1}\right)\,  \mathcal R^{-1},\label{ratesdc}
\end{align}
where
$\Lambda_1\ :=\ (\ell -1)L\mathcal R^2/4+2\ell d$.
\end{coroll}

The proof of the corollary is given in Section \ref{sec:optimize}. Bounds of the same order as in \eqref{ratesdc} hold if the constant $1/30$ is replaced by any other strictly positive constant. The specific value $1/30$ has been chosen in a somehow ad hoc way in order
to obtain relatively small constants in the prefactors.

\begin{remark}[Kinetic behaviour, Hamiltonian Monte Carlo]
Corollary \ref{cor:sdc} shows that by adjusting the 
friction coefficient $\gamma$ appropriately, one can obtain a kinetic 
lower bound for the contraction rate: If $\gamma$
is chosen such that the value of
$Lu\gamma^{-2}$ is a given constant, and the parameters $\beta$ and $\ell$ are fixed as well (i.e., $L\mathcal R^2$ is within a fixed range), then 
the lower bound for $c$ in \eqref{ratesdc}
is of order $\Omega
\left(\mathcal R^{-1}\right)$. This should be relevant for MCMC
methods based on discretizations of Langevin equations
\cite{neal, MR2681239, 2015arXiv151109382B},
because it indicates that by adjusting $\gamma$ appropriately,
one can improve on the diffusive order $O
\left(\mathcal R^{-2}\right)$ for the convergence rate to equilibrium.
\end{remark}

Before discussing the parameter dependence of the lower bounds for the contraction rate $c$  that have been stated above, 
we check the quality of the  bounds in the linear case and for 
drifts that are linear outside a ball:

\begin{example}[Linear drift]
Suppose that $U(x)=L|x|^2/2$. Then \eqref{Langevin} reads
\begin{eqnarray}
\label{LangevinG}\ \ dX_t &=&V_t\, dt,\qquad
 dV_t\ =\ -\gamma V_t\, dt\, -\,  u \, LX_t\, dt\,
+\, \sqrt{2\gamma u}\, dB_t.
\end{eqnarray}
Applying Corollary \ref{cor:sdc} with
$\beta =L\mathcal R^2$
and $\ell =1$ shows that for $Lu\gamma^{-2}\le 1/30$,
\begin{equation}
\label{cboundlinear}
c\ \ge\ \frac{\gamma}{205}\, \min\left(
(Lu\gamma^{-2})^2,\, \frac 12e^{-2d}\min\left( d^{1/2}Lu\gamma^{-2},(2d)^{-1/2}\right)\right) .
\end{equation} 
Lower bounds of a similar order can be derived
from Theorem \ref{thm:1} if $Lu\gamma^{-2}$ is bounded from above by a fixed constant.
On the other hand,  the linear Langevin equation \eqref{LangevinG} can be solved explicitly.
The solution is a Gaussian process. By \cite[Section 6.3]{MR3288096}, the $L^2$ spectral gap of the corresponding generator
is 
\begin{eqnarray}
\label{gap}
&&c_{gap} \  = \ \left(
1-\sqrt{(1-4Lu\gamma^{-2})^+}\right)\gamma/2,\qquad\mbox{and, in particular,}\\
\label{gap1} 
&&\gamma\, \min (1/4,Lu\gamma^{-2}) \ \le\  c_{gap}\ \le\  \gamma \, \min (1/2,2Lu\gamma^{-2}).
\end{eqnarray}
The spectral gap provides an upper bound for the 
contraction rate $c$. 
For example, for $d=1$ and $Lu\gamma^{-2}=1/30$, we obtain the lower bound
\begin{equation}
\label{cboundlineardim1}c\ \ge\ \frac{\gamma}{184500}\ \approx\ \frac{1}{33685}\, (Lu)^{1/2}
\end{equation}
for the contraction rate, whereas the upper bound given by the spectral gap is
$$c_{gap}\ =\ \frac 12\gamma\, (1-\sqrt{1-4/30})\ \approx\ \frac{\gamma}{29}\ \approx\ \frac{1}{5.3}(Lu)^{1/2}.$$
\end{example}

\begin{example}[Multi-well potentials, linear drift outside a ball]\label{ex:multiwell}
Assumption \ref{ass:2} with $\beta =L\mathcal R^2/2$ is satisfied for the one-dimensional double-well potential
\begin{equation}\nonumber
U(x)\ =\ \begin{cases}
L|x|^2/2 &\mbox{for }x\le\mathcal R/8,\\
-L(x-\mathcal R/4)^2/2+L\mathcal R^2/64 &\mbox{for }\mathcal R/8\le x\le3\mathcal R/8,\\
L(x-\mathcal R/2)^2/2&\mbox{for }x\ge 3\mathcal R/8,\\
\end{cases}
\end{equation}
and also for the triple-well potential $\tilde U(x)=U(|x|)$. Here,
for $Lu\gamma^{-2}\le 1/30$, Corollary \ref{cor:sdc} yields the lower bounds
\begin{align}\nonumber
\quad\ c \ &\ge \ \frac{\gamma}{410}\, \min\left(
(Lu\gamma^{-2})^2,\, \min\left( d^{1/2}Lu\gamma^{-2},(4d+L\mathcal R^2/4)^{-1/2}\right)\, e^{-4d-L\mathcal R^2/4}\right)\\
\ &\ge \ \frac{\sqrt{\beta u}}{75\,\mathcal R}\, 
\min\left(
(Lu\gamma^{-2})^2,\, \min\left( d^{1/2}Lu\gamma^{-2},(4d+L\mathcal R^2/4)^{-1/2}\right)\, e^{-4d-L\mathcal R^2/4}\right)\label{eq:lowerbounddw}
\end{align}
for the contraction rate. 
Again, lower bounds of similar order hold if $Lu\gamma^{-2}$ is bounded from above by a fixed constant. More generally,
we obtain corresponding bounds
if $U$ is a potential satisfying conditions \eqref{A0} and \eqref{A1}, and
there exist constants $\mathcal R\in\mathbb R_+$ and $a\in\mathbb R^d$
with $|a|\le\mathcal R/2$  such that $\nabla U(x)=L\, (x-a)$ for $|x|\ge\mathcal R$. The lower bound is of order $\Theta (\mathcal R^{-1})$ if $Lu\gamma^{-2}$
is fixed and $L\mathcal R^2$ is bounded from above by a fixed constant.
\end{example}

We stress that in low dimensions, we can obtain numerical values for our lower bounds that are in reach for current computer simulations. This is
quite remarkable because we have lost some factors during our estimates.
For high dimensions, our bounds deteriorate rapidly.

\subsection{Parameter dependence of lower bounds
for the contraction rate}

We now discuss the parameter dependence of the bounds derived above,
and we compare our results to previously derived bounds on convergence 
to equilibrium for kinetic Fokker-Planck equations.\smallskip

Let us first recall that 
the computation of the spectrum shows that in the linear case, there
are two different regimes, cf.\ \cite{MR3288096}. For $Lu\gamma^{-2}
\ge 1/4$ ({\em underdamped regime}), the spectral gap \eqref{gap} is a linear
function of $\gamma$. In this case, the friction coefficient $\gamma$
is so small that the rate of convergence to equilibrium is determined
by $\gamma$. Conversely, for $Lu\gamma^{-2}
\le 1/4$, the spectral gap is a decreasing function of $\gamma$.
In this regime, the rate of convergence to equilibrium is determined
by the transfer of noise from the $v$-variable to the $x$-variable.
If $\gamma$ increases then the noise is damped more strongly
before it can be transferred to the $x$-component, and hence the
rate of convergence decreases. In particular, the
spectral gap as a function of $\gamma$ has a sharp maximum
for $Lu\gamma^{-2}= 1/4$.\smallskip
\begin{figure}[h]
 \centering
\begin{center}
\includegraphics[height=4cm]{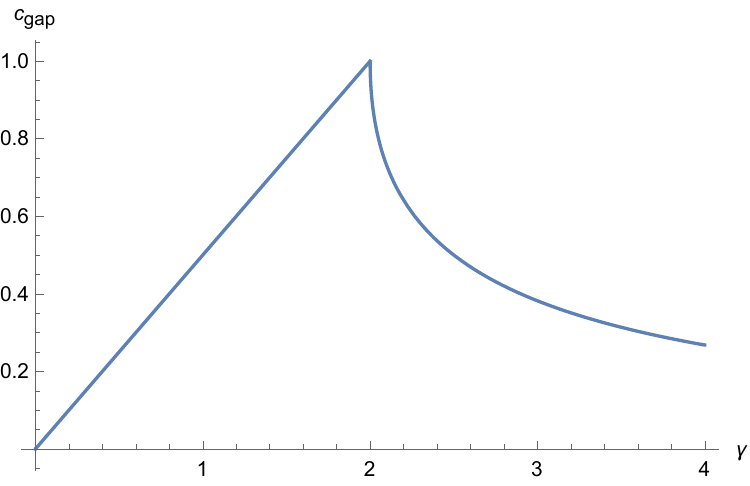} 
\end{center} 
\caption{Spectral gap for $U(x)=|x|^2/2$ and $u=1$}
 \end{figure}
 
Comparing \eqref{gap1} and 
\eqref{rate}, 
we see that our general lower bound for the contraction
rate contains similar terms as the bounds in \eqref{gap1}.
However, these terms are multiplied by constants that again
depend on the parameters $L,u$ and $\gamma$. We now discuss
the parameter dependence of the lower bounds in different regimes:
\smallskip\\
{\em $Lu\gamma^{-2}\to 0$ (overdamped case)}.
As $Lu\gamma^{-2}\to 0$, the lower bound in \eqref{ratesdc} is
of order $\Omega \left(\gamma (Lu\gamma^{-2})^2\right)$. This
differs from the order $\Theta \left(\gamma\, Lu\gamma^{-2}\right)$
of the spectral gap in the linear case by a factor $Lu\gamma^{-2}$.
\smallskip\\
{\em $Lu\gamma^{-2}$ fixed (kinetic case)}.
If the friction coefficient $\gamma$ is chosen such that the value of
$Lu\gamma^{-2}$ is a given constant, and the parameters $\beta$ and $L\mathcal R^2$ are fixed as well, then 
the lower bound in \eqref{ratesdc} is of order $\Omega
\left(\mathcal R^{-1}\right)$. 
\smallskip\\
{\em $Lu\gamma^{-2}\to \infty$ (underdamped case)}.
For large values of $Lu\gamma^{-2}$, our bounds for the
contraction rate in Lemma \ref{lem:sdc} are considerably worse
than the order $\Theta (\gamma )$ of the spectral gap in the linear
case. This is not surprising, because for small $\gamma$ it is not
clear, how the contractive term $-\gamma V_t$ in \eqref{Langevin}
can make up for the term $-u\nabla U(X_t)$ with a potential $U$
that may be locally very non-convex. 
\smallskip\\
{\em $L\mathcal R^{2}\to \infty $}.
For large values of $L\mathcal R^2$, the lower bound in
\eqref{ratesdc} degenerates exponentially in this parameter. This
is natural because $U$ could be a double-well
potential with valleys of depth $\Theta (L\mathcal R^2)$, see the example above.
\smallskip\\
{\em $d\to\infty$}. 
Our bounds depend exponentially on the dimension. In the general
setup considered here, this is unavoidable. An important open 
question is whether a better dimension dependence can be obtained
for a restricted class of models. For overdamped Langevin diffusions,
corresponding results have been obtained for example in 
\cite{EberlePTRF,zimmer16,zimmerPhD}.
\medskip
 
Let us now set our results in relation to explicit bounds for convergence to 
equilibrium of kinetic Fokker-Planck equations that have been obtained
in \cite{MR2034753,MR2562709,MR3324910} by analytic methods.
These results are not directly comparable, because they quantify
convergence to equilibrium in different distances (e.g.\ in weighted $L^2$ or
Sobolev norms, or in relative entropy). Moreover, the constants have not
always been tracked as precisely as here. 
Nevertheless, it seems plausible to compare the orders of the convergence rates. Already in \cite[Theorem 0.1]{MR2034753}, H\'erau and Nier have derived a nice explicit bound on the convergence rate in weighted 
Sobolev spaces under quite general conditions. However, even in the linear case, this bound is far
from sharp in the overdamped regime and at the boundary between
the overdamped and the underdamped regime. In particular, it 
seems not to be able to recover kinetic behavior for properly
adjusted friction coefficients. Most of the more recent works are
based on extensions of Villani's
hypocoercivity approach \cite{MR2562709}. In particular,
explicit bounds are given in \cite[Section 7]{MR2562709} in $H^1$ norms,
corresponding bounds in $L^2$ norms can be deduced from \cite{MR3324910}, and recently some Wasserstein bounds have been 
derived in \cite{MR3488535}. However, we have not been able to recover a similar
behaviour as above for the bounds on the convergence rates in these results.

\subsection{Outline of the proof}
To prove our main result we proceed in the following way. In Section
\ref{sec:coupling}, we precisely define the coupling that we consider. 
Having introduced both the coupling and the underlying distance function,
we study average contraction properties of the coupling distance $\rho_t$
by standard methods from stochastic analysis. To this end, we compute the
semimartingale decomposition of $e^{ct}\rho_t$ for a given constant $c>0$.
Exponential contractivity with rate $c$ holds if the resulting drift term is negative.
In Section \ref{sec:contrregion} we analyse under which conditions on the parameters this holds true in different regions of the state space for the coupling process. 
By \eqref{B6}, contractivity on the hyperplane where synchronous coupling is applied can only be expected provided $\alpha >Lu\gamma^{-2}$. This motivates setting $\alpha =(1+\eta )Lu\gamma^{-2}$ with $\eta >0$ in \eqref{alpha}.
To ensure that reflection coupling yields contractivity outside of this hyperplane,
one has to choose $f$ sufficiently concave. Intuitively, by applying a sufficiently 
concave function, we can turn the submartingale $r_t=\alpha |Z_t|+|Q_t|$ into a 
supermartingale. This approach has been used in several previous works
\cite{MR2843007,EberlePTRF,2016arXivE}, and carrying it out in an optimized way leads to the choice of $f$ given by \eqref{f1}, \eqref{phi1} and \eqref{g1}.
Having fixed $f$, one can now see that contractivity on the hyperplane holds if
$c$ is sufficiently small depending on $\eta$, cf.\ \eqref{c2}. Moreover, the Lyapunov condition implies contractivity at large distances if \eqref{c3} holds.
As a consequence, we obtain a global contraction result  with a contraction rate $c$ depending on the parameters in the definition of the metric, see Theorem \ref{thm:D}.
The final step of the proof of the main result then consists in choosing these parameters in order to maximize the resulting contraction rate approximately. This is carried out in Section \ref{sec:optimize}. Roughly, the constants $\alpha$ and $\epsilon$ in \eqref{rDist} and
\eqref{rho} are chosen sufficiently small such that the effects of the distortion of the metric do not destroy the contraction properties at small distances, but otherwise as large as possible. This is ensured by \eqref{42}, \eqref{42a} and 
\eqref{44}. With this choice of constants, the main result then follows from 
Theorem \ref{thm:D}.

\section{Coupling and evolution of coupling distance}\label{sec:coupling}

We fix a positive constant $\xi$. Eventually, we will consider the
limit $\xi\downarrow 0$.
In order to couple two solutions of \eqref{Langevin}, we consider 
the following SDE on $\mathbb R^{2d}\times\mathbb R^{2d}$:
\begin{eqnarray}\nonumber
dX_t &=&V_t\, dt,\\
\nonumber 
dV_t&=&-\gamma\,V_t\, dt\, \,-\,  u \, \nabla U(X_t)\, dt\,
+\, \sqrt{2\gamma u}\;rc(Z_t,W_t)\, dB_t^{rc}\\&&\label{coupling}
\qquad \qquad +\, \sqrt{2\gamma u}\;sc(Z_t,W_t)\, dB_t^{sc},\\
\nonumber dX_t^{\prime} &=&V_t^{\prime}\, dt,\\
\nonumber
dV_t^{\prime}&=&-\gamma\,V_t^{\prime}\, dt\, -\,  u \, \nabla U(X_t^{\prime})\, dt\,
+\, \sqrt{2\gamma u}\;rc(Z_t,W_t)\,(I_d-2e_te_t^T)\, dB_t^{rc}\\&&\nonumber\qquad \qquad
+\, \sqrt{2\gamma u}\;sc(Z_t,W_t)\, dB_t^{sc}.
\end{eqnarray} 
Here $B^{rc}$ and $B^{sc}$ are independent Brownian motions,
\begin{eqnarray}
\label{ZWQ}
Z_t &=& X_t-X_t^\prime ,\quad W_t\ =\ V_t-V_t^\prime ,\quad
Q_t\ =\ Z_t+\gamma^{-1}W_t,\\
\label{et}e_t& =& Q_t/|Q_t|\quad\mbox{if }Q_t\neq 0,\quad\mbox{and}\quad
e_t\ =\ 0\quad\mbox{if }Q_t=0.
\end{eqnarray}
Moreover, $rc, sc
:\mathbb R^{2d}\to [0,1]$ are Lip\-schitz continuous functions such that
$rc^2+sc^2\equiv 1$,
\begin{align}
\label{rc0}rc(z,w) \ &=\ 0\quad\mbox{if }z+  \gamma^{-1}w=0\quad\mbox{or}\quad \alpha |z|+|z+  \gamma^{-1}w|\ge R_1+\xi ,\\
\label{rc1}rc(z,w) \ &=\ 1\quad\mbox{if }|z+ \gamma^{-1}w|\ge\xi\quad\mbox{and}\quad \alpha |z|+|z+  \gamma^{-1}w|\le R_1 .
\end{align}
The values of the constants $\alpha ,R_1\in (0,\infty )$ 
will be fixed later.
Note that by \eqref{rc0}, $rc(Z_t,W_t)e_te_t^T$ is a Lipschitz continuous function of $(X_t,V_t,X_t^\prime ,V_t^\prime )$. Therefore, 
existence and uniqueness of the coupling process holds by It\^o's theorem.
Moreover, by L\'evy's characterization, for any solution of \eqref{coupling},
\begin{eqnarray*}
B_t&=&\int_0^t \,rc(Z_s,W_s)\, dB_s^{rc}\,
+\, \int_0^t sc(Z_s,W_s)\, dB_s^{sc}\qquad\mbox{and}\\
B_t^\prime &=&\int_0^t \,rc(Z_s,W_s)\,(I_d-2e_se_s^T)\,  dB_s^{rc}\,
+\, \int_0^t sc(Z_s,W_s)\, dB_s^{sc}
\end{eqnarray*}
are again Brownian motions. Thus \eqref{coupling} defines indeed
a coupling of two solutions of \eqref{Langevin}. For $rc\equiv 1$
and $sc\equiv 0$, the coupling is a {\em reflection coupling}, whereas for $rc\equiv 0$ and $sc\equiv 1$ it is a
{\em synchronous coupling}. By \eqref{rc0}
and \eqref{rc1}, we 
choose $rc$ and $sc$ such that the synchronous coupling is applied when $r_t=\alpha |Z_t|+|Q_t|$ is large or $Q_t$ is close to $0$, and the reflection coupling is applied otherwise. Ideally, we would like to choose $rc(Z_t,W_t)=1_{\{ Q_t\neq 0,\, r_t<R_1\} }$. Since the indicator function is not continuous, we consider Lipschitz continuous approximations instead.\medskip
 
The processes $(Z_t),(W_t)$ and $(Q_t)$ satisfy the following equations:
\begin{equation}
\label{Zt}dZ_t \ =\ W_t\, dt\ =\ \gamma Q_t\, dt\, -\, \gamma Z_t\, dt,
\end{equation}
$$\qquad dW_t = -\gamma W_t\, dt -  u  (\nabla U(X_t)-\nabla U(X_t^\prime ))\, dt
+ \sqrt{8\gamma u}\,rc(Z_t,W_t)e_te_t^T\,  dB_t^{rc},$$
\begin{equation}
\label{Qt}dQ_t =  - u\gamma^{-1}  (\nabla U(X_t)-\nabla U(X_t^\prime ))\, dt
+ \sqrt{8\gamma^{-1} u}\,rc(Z_t,W_t)\,e_te_t^T\,  dB_t^{rc}.
\end{equation}
Note that in particular, $Z_t$ is contractive when $Q_t=0$, and
$Q_t$ would be a local martingale if $U$ would be constant.
We set
\begin{eqnarray}
\label{r} r_t&:=&r((X_t,V_t),(X_t^\prime ,V_t^\prime ))\ =\  \alpha |Z_t|\, +\,
|Q_t|,\qquad\mbox{and}\\
\label{rhoa} \rho_t&:=&\rho ((X_t,V_t),(X_t^\prime ,V_t^\prime ))\ =\ 
f(r_t)\cdot G_t,\qquad\mbox{where}\\
\label{G} G_t&:=&1\, +\, \epsilon \lyp(X_t,V_t)\, +\, \epsilon
\lyp(X_t^\prime ,V_t^\prime ).
\end{eqnarray}

\begin{lemma}\label{lem:B}
Let $c,\epsilon \in (0,\infty )$, and suppose that $f:[0,\infty )\to [0,\infty )$
is continuous, non-decreasing, concave, and $C^2$ except for
finitely many points. Let
\begin{equation}
\label{alpha}\alpha \ =\ (1+\eta )Lu\gamma^{-2} 
\end{equation}
for some constant $\eta\in (0,\infty )$. Then
\begin{equation}
\label{rhot}
e^{ct}\rho_t\ \le  \ \rho_0\, +\, \gamma\, \int_0^te^{cs}K_s\, ds\, +\, M_t\qquad\mbox{for any }t\ge 0,
\end{equation}
where $(M_t)$ is a continuous local martingale, and
\begin{align}
\label{K} 
K_t \ &=\  4u\gamma^{-2}\, rc(Z_t,W_t)^2f''(r_t)\, G_t\\ \nonumber \ & +\, (\alpha
|Q_t|-\frac{\eta}{1+\eta}\alpha |Z_t|)\, f_-^\prime (r_t)\, G_t\, +\, 4\epsilon
\max(1,(2\alpha)^{-1})\,rc(Z_t,W_t)^2\, r_t f'_{-}(r_t)
	\\
\ &\nonumber  +\, \left( 2(d+A)-\lambda \lyp(X_t,V_t)-\lambda \lyp(X'_t,V'_t)\right)
\epsilon f(r_t)\, +\, \gamma^{-1}c\, f(r_t)\, G_t .
\end{align}
\end{lemma}

\begin{proof}[Proof of Lemma \ref{lem:B}]
By \eqref{Zt}, the paths of the process $(Z_t)$ are almost surely
continuously differentiable with derivative $dZ/dt=-\gamma Z+\gamma Q$. Therefore, $t\mapsto |Z_t|$ is almost surely
absolutely continuous with
\begin{eqnarray*}
\frac{d}{dt}|Z_t|& =& \frac{Z_t}{|Z_t|	}\cdot (-\gamma Z_t+\gamma Q_t)\qquad\mbox{for a.e.\ }t\mbox{ such that }Z_t\neq 0,\mbox{ and}\\
 \frac{d}{dt}|Z_t|& \le & \gamma |Q_t|\qquad\mbox{for a.e.\ }t\mbox{ such that }Z_t= 0.
\end{eqnarray*}
In particular, we obtain
\begin{equation}
\label{B1}\frac{d}{dt}|Z_t|\ \le\ -\gamma |Z_t|+\gamma |Q_t|\qquad\mbox{for
a.e.\ } t\geq 0.\\
\end{equation}

The process $(Q_t)$ satisfies the SDE \eqref{Qt}. Notice that by \eqref{rc0},
the noise coefficient vanishes if $Q_t=0$. Therefore, similarly to Lemma 4 in
\cite{EberlePTRF}, we can apply It\^o's formula to conclude that almost surely,
\begin{eqnarray}\label{B2}
	|Q_t| &=& |Q_0|\, +\, A_t^Q\, +\, \tilde{M}_t^Q\qquad\text{for all } t\ge 0,
\end{eqnarray}
where $(A_t^Q)$ and $(\tilde{M}_t^Q)$ are the absolutely continuous process and
the martingale given by
\begin{eqnarray*}
\label{B3} A_t^Q &=& -u\gamma^{-1}\, \int_0^t e_s^T\cdot(\nabla U(X_s)-\nabla
U(X'_s))\, ds, \\
\label{B4} \tilde{M}_t^Q &=& \sqrt{8u\gamma^{-1}}\, \int_0^t rc(Z_s,W_s)\,
e_s^T\, dB_s^{rc}.
\end{eqnarray*}
Notice that there is no It\^o correction, because $\partial^2_{q/|q|} |q|=0$ for
$q\not=0$ and the noise coefficient vanishes for $Q_t=0$. By \eqref{B1},
\eqref{B2} and the Lipschitz condition on $\nabla U$, we conclude that
$r_t=\alpha |Z_t|+|Q_t|$ has the semimartingale decomposition
\begin{eqnarray*}
\label{B5} r_t &=& |Q_0|\, +\, \alpha |Z_t|\, +\, A_t^Q\, + \, \tilde{M}_t^Q,
\end{eqnarray*}
where $t\mapsto \alpha |Z_t| + A_t^Q$ is almost surely absolutely continuous
with derivative
\begin{align}\label{B6}
\frac{d}{dt} (\alpha |Z_t|\, +\, A_t^Q) \ &\leq\  \left(
(Lu\gamma^{-2}-\alpha)|Z_t|+\alpha |Q_t|\right)\,\gamma\qquad\text{for a.e.\
} t\geq 0.
\end{align}
Since by assumption, $f$ is concave and piecewise $C^2$, we can now apply the 
It\^o-Tanaka formula to $f(r_t)$. Let $f'_{-}$ and $f''$ denote the left-sided
first derivative and the almost everywhere defined second derivative.
Notice that the generalized second derivative of $f$ is a signed measure
$\mu_f$ such that $\mu_f(dr)\leq f''(r)\, dr$. We obtain a semimartingale
decomposition
\begin{eqnarray}\label{B7}
	e^{ct}\,f(r_t) &=& f(r_0)\, +\, \tilde{A}_t\, + \, \tilde{M}_t
\end{eqnarray}
with the martingale part
\begin{eqnarray}\label{B8}
	\tilde{M}_t &=& \sqrt{8u\gamma^{-1}}\, \int_0^t e^{cs}\, f'_{-}(r_s)\,
	rc(Z_s,W_s)\, e_s^T\, dB_s^{rc}
\end{eqnarray}
and a continuous finite-variation process $(\tilde{A}_t)$ satisfying 
\begin{align} \label{B9}
	d\tilde{A}_t \ &=  \left(c\, f(r_t)\, +\,(
(Lu\gamma^{-2}-\alpha)|Z_t|+\alpha |Q_t|)\, \gamma\,f'_{-}(r_t) \right)\, e^{ct}\, dt
\\ \ & + \left(4u\gamma^{-1}\, rc(Z_t,W_t)^2\, f''(r_t) \right)\, e^{ct}\,
dt.\nonumber
\end{align}
Next, we consider the time evolution of the process $G_t=1 + \epsilon
\lyp(X_t,V_t) + \epsilon \lyp(X'_t,V'_t)$. An application of It\^o's formula
shows that by \eqref{coupling},
\begin{eqnarray}
	\nonumber dG_t &=& \epsilon (\mathcal L \lyp)(X_t,V_t)\, dt + \epsilon (\mathcal
	L \lyp)(X'_t,V'_t)\, dt\\
	&& +\,  \epsilon \sqrt{2u\gamma} (\nabla_v \lyp(X_t,V_t)) - \nabla_v
	\lyp(X'_t,V'_t))^T rc(Z_t,W_t)\,e_te_t^T\,   dB_t^{rc}\label{B10}
\\
\nonumber && +\, \epsilon \sqrt{2u\gamma} (\nabla_v \lyp(X_t,V_t)) + \nabla_v	\lyp(X'_t,V'_t))^T rc(Z_t,W_t)\, (I_d-e_t e_t^T) \, dB_t^{rc}
\\
\nonumber	&& +\, \epsilon \sqrt{2u\gamma} (\nabla_v \lyp(X_t,V_t)) + \nabla_v
	\lyp(X'_t,V'_t))^T sc(Z_t,W_t) \, dB_t^{sc}
\end{eqnarray}
Hence by \eqref{B7}, \eqref{B10}, the It\^o product rule and \eqref{B8},
we obtain the semimartingale decomposition
\begin{eqnarray}\label{B10a}
	e^{ct}\rho_t &=& e^{ct}\, f(r_t)\, G_t \, = \, \rho_0 \, + \, M_t \, + \, A_t,
\end{eqnarray}
where $(M_t)$ is a continuous local martingale, and
\begin{eqnarray}\nonumber
	dA_t &=& G_t\, d\tilde{A}_t \, + \epsilon\, e^{ct}\, f(r_t)\,  \left((\mathcal
	L \lyp)(X_t,V_t)\, +\, (\mathcal L \lyp)(X'_t,V'_t)\right)\, dt\\
	&& + \, 4\epsilon u \, e^{ct}\,  f'_{-}(r_t)\, rc(Z_t,W_t)^2 (\nabla_v
	\lyp(X_t,V_t)) - \nabla_v \lyp(X'_t,V'_t))^T e_t\, dt. \label{B11}
\end{eqnarray}
Here we have used that $B^{rc}$ and $B^{sc}$ are independent Brownian motions in
$\mathbb{R}^d$. Now recall that by Lemma \ref{lem:L}, 
\begin{eqnarray} \label{B12}
	\mathcal L \lyp &\leq& (d\,+\,A\,-\,\lambda \lyp)\,\gamma.
\end{eqnarray} Furthermore, a simple computation shows that by
\eqref{H}, 
\begin{eqnarray}\nonumber
\lefteqn{	|\nabla_v \lyp(X_t,V_t) - \nabla_v \lyp(X'_t,V'_t)| = |\frac{\gamma}{2u}  (X_t+ \frac 2{\gamma}V_t-X'_t-\frac 2{\gamma}V'_t)|}\\
\label{B13} &\leq  u^{-1}\gamma \, (|Q_t|\,+\,|Z_t|/2) \, \leq \,
u^{-1}\gamma \max(1,(2\alpha)^{-1})\, r_t.
\end{eqnarray}
By combining \eqref{B11}, \eqref{B9}, \eqref{B12} and
\eqref{B13} we finally obtain
$
	dA_t \leq \gamma \, e^{ct}\, K_t\, dt$, where
\begin{eqnarray*}
	K_t &= &\gamma^{-1} c f(r_t)G_t\, +\, ((Lu\gamma^{-2}-\alpha)|Z_t|+\alpha
	|Q_t|)f'_{-}(r_t)G_t\\
	&&+4u\gamma^{-2}rc(Z_t,W_t)^2 f''(r_t)G_t
	\\ &&+\epsilon f(r_t)\left(2(d+A)-\lambda \lyp(X_t,V_t)-\lambda
	\lyp(X'_t,V'_t)\right)
	\\&&+ 4\epsilon \max(1,(2\alpha)^{-1}) r_t f'_{-}(r_t)
	rc(Z_t,W_t)^2.
\end{eqnarray*}
The assertion now follows
from \eqref{B10a} by setting $\alpha=(1+\eta)Lu\gamma^{-2}$.
\end{proof}

\section{Contractivity in different regions}\label{sec:contrregion}
As above suppose that $\alpha =(1+\eta )Lu\gamma^{-2}$ for some
$\eta\in (0,\infty )$. We will choose $\epsilon $ and $c$ such that
\begin{equation}
\label{epsc}(d+A)\, \epsilon\ =\ 4\, c/\gamma .
\end{equation}
This choice guarantees in particular that the terms $2(d+A)\epsilon$
and $\gamma^{-1}c$ in \eqref{K} are of comparable order.

We consider a coupling as introduced in Section \ref{sec:coupling}.
In order to make $\rho_t$ a contraction on average, we have to
choose the parameters such that $K_t\le 0$. This will be achieved
up to an error term which vanishes as $\xi\downarrow 0$. In order 
to bound $K_t$, we distinguish between different regions of the
state space:\medskip

{\em (i) $|Q_t|\ge\xi$ and $r_t\le R_1$.} Here by \eqref{rc1},
$rc(Z_t,W_t)=1$. Therefore, by \eqref{K} and \eqref{epsc},
and since $|Q_t|\le r_t$ and $G_t\ge 1$, 
\begin{eqnarray}
\nonumber
K_t &\le & 4u\gamma^{-2}\, f''(r_t)\, G_t\, +\, ((1+\eta )Lu\gamma^{-2} +
4\epsilon \max(1,(2\alpha)^{-1}))\, r_t\, f_-^\prime(r_t)\, G_t \\ && \, +\,
9\, \gamma^{-1}c\, f(r_t)\, G_t .\nonumber
\end{eqnarray}
Similarly to \cite{EberlePTRF},
we can now ensure that $K_t\le 0$ by choosing $f$ appropriately. We set
\begin{align}
\label{f1}f(r)\ &=\ \int_0^{r\wedge R_1}\varphi (s)\, g(s)\, ds,\qquad
\mbox{where}\\
\label{phi1}\varphi (s) \ &=\  \exp\left( -(1+\eta )Ls^2/8-\gamma^2u^{-1}\epsilon\max(1,(2\alpha)^{-1})s^2/2\right) ,\quad\mbox{and}\\
\label{g1}g(r) \ &=\  1\, -\, \frac 94c\gamma u^{-1}\int_0^r\Phi (s)
\varphi (s)^{-1}\, ds\quad\mbox{with }\ \Phi (s)=\int_0^s\varphi (x)\, dx.
\end{align}
Then $4u\gamma^{-2}\varphi^\prime (r_t)+((1+\eta )Lu\gamma^{-2}
+4\epsilon\max(1,(2\alpha)^{-1}))r_t\varphi (r_t)=0$, and hence
\begin{equation}
\label{K1}K_t \ \le \ 4u\gamma^{-2}\varphi (r_t)g^\prime (r_t)G_t\, +
\, 9\gamma^{-1}c\, \Phi (r_t)G_t\ \le\ 0.
\end{equation}
In order to ensure $g(r)\ge 1/2$ for $r\le R_1$, we have to assume
\begin{equation}
\label{c1}c\ \le\ \frac 29\, u\gamma^{-1}\left/ \int_0^{R_1}\Phi (s)\varphi (s)^{-1}\, ds\right. \, .
\end{equation}

{\em (ii) $|Q_t|<\xi$ and $r_t\le R_1$.} 
In this region, there is a transition from
reflection coupling to synchronous coupling, which is applied for $Q_t=0$. Hence we can not rely on the 
additional contraction properties gained by applying 
reflection coupling and choosing $f$ 
sufficiently concave. Instead, however, we can use that
\begin{equation}
 \label{31} |Q_t|\le\xi\qquad\mbox{and}\qquad
 \alpha |Z_t|\ =\ r_t-|Q_t|\ \ge\ r_t-\xi .
 \end{equation} 
By the choice of $f$ and since $g\ge 1/2$, we obtain, similarly to Case (i),
\begin{eqnarray}
\nonumber
K_t &\le & \left(4u\gamma^{-2}\, f''(r_t)\, G_t\, +\,4\epsilon\max(1,(2\alpha)^{-1})r_tf_-^\prime(r_t)\right)\,  rc(Z_t,W_t)^2\\
\nonumber &&-\tfrac{\eta}{1+\eta}r_t\, f_-^\prime (r_t)\, G_t\, +\, 
(1+\alpha )\xi f_-^\prime (r_t)\, G_t\, +\, 
9 \gamma^{-1}c\, f(r_t)\, G_t \\
\nonumber &\le &-\tfrac 12\, \tfrac{\eta}{1+\eta}\, r_t\, \varphi (r_t)G_t\, +\, 9 \gamma^{-1}c\, \Phi (r_t)\, G_t\, +\, (1+\alpha )\xi G_t.
\end{eqnarray}
In order to ensure that the upper bound converges to $0$ as $\xi\downarrow 0$, we assume
\begin{equation}
\label{c2}c \ \le\ \frac{1}{18}\, \gamma\, \frac{\eta}{1+\eta}\, 
\inf_{r\in (0,R_1]}\frac{r\varphi (r)}{\Phi (r)}.
\end{equation}
Under this condition, we obtain
\begin{equation}
\label{K2}K_t\ \le\ (1+\alpha )\, \xi\, G_t.
\end{equation}

{\em (iii) $r_t>R_1$.} Here $f'_-(r_t)=0$. Hence by
\eqref{K}, \eqref{G}, \eqref{epsc} and \eqref{H2}, 
\begin{align}
\label{K3}
K_t \ &=\left[ 2(d+A)+\gamma^{-1}c\epsilon^{-1}-(\lambda-\gamma^{-1}c)(\lyp(X_t,V_t)+\lyp(X'_t,V'_t))\right]\, \epsilon\, f(r_t)\\\nonumber
\ &\le \ \left[ \tfrac 94(d+A)-\tfrac{15}{16}\lambda\, (\lyp(X_t,V_t)+\lyp(X'_t,V'_t))\right]\, \epsilon\, f(r_t)\ \le\ 0,
\end{align}
provided we assume
\begin{equation}
\label{c3}c\ \le\ \gamma\lambda /16.
\end{equation}

Notice that the functions $\varphi$, $g$ and $f$ in \eqref{phi1},
\eqref{g1}, \eqref{f1}, and the constants $\alpha$, $\epsilon$
determining $\rho$ do not depend on the value of $\xi$. Therefore,
by combining Lemma \ref{lem:B}, \eqref{K1}, \eqref{K2} and
\eqref{K3}, we obtain the following general result:

\begin{thm}\label{thm:D}
Let $\alpha =(1+\eta )Lu\gamma^{-2}$ for some $\eta\in (0,\infty )$,
and let 
$\epsilon =4\gamma^{-1}c/(d+A)$ for a positive constant $c$ such that
\begin{equation}
\label{c0}c\ \le\ \gamma\, \min\left(
\frac 29u\gamma^{-2}\left/ \int_0^{R_1}\frac{\Phi (s)}{\varphi (s)}\, ds\, ,\; \frac{1}{18}\, \frac{\eta}{1+\eta}\,\inf_{s\in (0,R_1]}
\frac{s\varphi (s)}{\Phi (s)}\, ,\; \frac{1}{16}\lambda\right.\right).
\end{equation}
Moreover, let $f:[0,\infty)\to [0,\infty )$ be defined by \eqref{f1},
\eqref{phi1} and \eqref{g1}. Then for any $t\ge 0$ and for any 
probability measures $\mu ,\nu$ on $\mathbb R^{2d}$,
\begin{equation}
\label{Wcontr}\mathcal W_\rho (\mu p_t,\nu p_t)\ \le\ e^{-ct}\,
\mathcal W_\rho (\mu ,\nu ).
\end{equation}
\end{thm}

\begin{proof}
Let $\Gamma$ be a coupling of two probability measures $\mu$ and $\nu$ on
$\mathbb{R}^d$ such that $\mathcal W_\rho(\mu,\nu)<\infty$. We consider the
coupling process $((X_t,V_t),(X'_t,V'_t))$ introduced in Section
\ref{sec:coupling} with initial law $((X_0,V_0),(X'_0,V'_0))\sim\Gamma$.
By \eqref{c0}, the conditions \eqref{c1}, \eqref{c2}, \eqref{c3} are satisfied.
Therefore, in each of the cases (i), (ii) and (iii) considered above, we obtain
$K_t\leq (1+\alpha)\xi G_t$. Therefore we apply Lemma \ref{lem:B}. By taking
expectations in \eqref{rhot}, evaluated at localizing stopping times
$T_n\uparrow t$ and applying Fatou's lemma as $n\rightarrow\infty$, we obtain
\begin{eqnarray}\label{approxcontr}
	\mathbb{E}[\rho_t]&\leq& e^{-ct}\, \mathbb E[\rho_0]\, +\,
	\gamma (1+\alpha)\xi \int_0^t e^{c(s-t)}\, \mathbb E[G_s]\, ds
\end{eqnarray}
for any $\xi>0$ and $t\geq 0$. Note that the coupling process and the coupling 
distance $\rho_t$ still depend on the value of $\xi$. On the other hand,
the expectation of $G_s$ does not depend on $\xi$. Indeed, by \eqref{G},
\begin{eqnarray}
\label{EGt}	\mathbb E[G_s] &=& \mathbb E[1\,+\,\epsilon \lyp(X_s,V_s)\,+\,\epsilon
\lyp(X'_s,V'_s)]\\
&=& 1\, +\,\epsilon \int p_s \lyp\, d\mu \,+\,\epsilon \int p_s \lyp\, d\nu.  \nonumber
\end{eqnarray}
Since $\mathcal W_{\rho}(\mu,\nu)<\infty$, both $\int \lyp \,d\mu$ and $\int
\lyp \,d\nu$ are finite. Therefore, by the Lyapunov condition in Lemma \ref{lem:L},
the expectation in \eqref{EGt} is finite, too.

Since $((X_t,V_t),(X'_t,V'_t))$ is a coupling of $\mu p_t$ and $\nu p_t$, we
have $\mathcal W_\rho(\mu p_t, \nu p_t)\leq \mathbb E[\rho_t]$ for any $\xi>0$.
Moreover, $\mathbb E[\rho_0]=\int \rho\, d\Gamma$. Hence by applying
\eqref{approxcontr} and taking the limit $\xi\downarrow 0$, we obtain
\begin{eqnarray}\label{Wcontr1}
	\mathcal W_\rho(\mu p_t, \nu p_t) &\leq & e^{-ct}\, \int  \rho\, d\Gamma\qquad
	\text{for any } t\geq 0.
\end{eqnarray}
The assertion follows since \eqref{Wcontr1} holds for an arbitrary coupling of
$\mu$ and $\nu$.
\end{proof}

\section{Choice of the constants}\label{sec:optimize}
The function 
\begin{eqnarray}\label{40}
	\varphi(s)&=&
	\exp\left(-(1+\eta)Ls^2/8-u^{-1}\gamma^2\epsilon\max(1,(2\alpha)^{-1})\, s^2/2
	\right)
\end{eqnarray}
determining the metric in Theorem \ref{thm:D} still depends on the values of the
constants $\eta\in(0,\infty)$, $\alpha=(1+\eta )Lu\gamma^{-2}$ and
$\epsilon\in(0,\infty)$. In order to prove our main result, we 
now choose explicit values for $\eta$ and $\epsilon$ (and hence for $\alpha$
and $c$). We first discuss how to choose these constants in order to optimize
the resulting bound for the contraction rate in \eqref{c0} approximately, and
then we apply Theorem \ref{thm:D} with the chosen constants. A first condition
we want to be satisfied is that $\varphi(s)$ is bounded below by the 
minimal possible value $\exp(-Ls^2/8)$ up to a multiplicative constants.
Therefore, we choose $\eta$ and $\epsilon$ such that
\begin{eqnarray}
\label{41} \varphi(s) &\geq& e^{-2}\, \exp(-Ls^2/8) \qquad\text{for any }
s\in[0,R_1].
\end{eqnarray}
By \eqref{40}, this bound holds true if 
\begin{eqnarray}\label{42}
	\eta&=& 8/(LR_1^2)=\Lambda^{-1} \qquad\text{or, equivalently,} \\
\label{42a} \alpha &=& (1+\Lambda^{-1})Lu\gamma^{-2} \, = \, (L+8R_1^{-2})
u\gamma^{-2}, \qquad \text{and} \\
\label{44} \epsilon &\leq& 2 \min(1,2\alpha)\, u \gamma^{-2} R_1^{-2}.
\end{eqnarray}
We recall from \eqref{R1} that
\begin{align}
	\label{R0} R_1 \,=\, \sqrt{(1+\alpha)^2+\alpha^2}\, R_0, \quad\text{where}\quad
	R_0\,
	:=\, u^{1/2}\gamma^{-1}\left(\frac{96(d+A)}{5\lambda(1-2\lambda)}\right)^{1/2}.
\end{align}
In particular, $R_1^{-2}$ is a decreasing function of $\alpha$, and hence there
are unique values $\alpha,\eta\in(0,\infty)$ such that \eqref{42a} and
\eqref{42} are satisfied. We fix $\eta$ and $\alpha$ correspondingly, and then
we choose $\epsilon>0$ such that \eqref{44} is satisfied. By \eqref{epsc} and
\eqref{R0} the condition \eqref{44} on $\epsilon$ is equivalent to the following 
additional constraint on the constant $c$ in Theorem \ref{thm:D}:
\begin{eqnarray}\label{45}
	c &=& \gamma \epsilon (d+A)/4  \ \leq\  \min(1,2\alpha)\, u \gamma^{-1} (d+A)
	R_1^{-2}/2 \\ & = & \frac{5}{192}
	\frac{\min(1,2\alpha)}{(1+\alpha)^2+\alpha^2}\, \gamma \lambda (1-2\lambda).\nonumber
\end{eqnarray}
Since the last expression is smaller than $\gamma \lambda/16$, we see that 
Theorem \ref{thm:D} applies with parameters $\eta$ and $\epsilon$ satisfying
\eqref{42} and \eqref{44} whenever $c$ is smaller than
\begin{align}\label{46}
	\gamma \min\left( \frac{5}{192} \frac{\min(1,2\alpha)\lambda
	(1-2\lambda)}{(1+\alpha)^2+\alpha^2},
\frac {2u}{9\gamma^{2}}\left/ \int_0^{R_1}\frac{\Phi }{\varphi }\right. ,\, \frac{1}{18} \frac{\eta}{1+\eta}\inf_{s\in (0,R_1]}
\frac{s\varphi (s)}{\Phi (s)} \right).
\end{align}
We can now complete the proof of Theorem \ref{thm:1} by deriving an explicit
lower bound for the right hand side of \eqref{46}.

\begin{proof}[Proof of Theorem \ref{thm:1}]
We fix $\eta$ and $\alpha$ as in \eqref{42}, \eqref{42a} and we choose
$\epsilon>0$ such that \eqref{44} holds. For these parameters we give explicit
lower bounds for the three terms in the mininum in \eqref{46}: 
\par\medskip
1. By \eqref{42a}, \eqref{R0} and \eqref{lambda1},
	\begin{eqnarray}\nonumber
		Lu\gamma^{-2} &\leq& \alpha \, \leq \, (L+8R_0^{-2}) u\gamma^{-2} \, \leq \,
		Lu\gamma^{-2}\, +\, \frac{5}{12} \lambda/(d+A)\\\label{49}
		&\leq & (1+\frac{5}{6}(d+A)^{-1}) Lu\gamma^{-2} \, \leq \, \frac{11}{6}
		Lu\gamma^{-2}.
	\end{eqnarray}
	Therefore, and since $\lambda \leq 1/4$, we obtain
	%
%
\begin{align}
	\frac{\min(1,2\alpha)}{(1+\alpha)^2+\alpha^2}\, &\ge\,
	\min\left(\frac{4}{5}\alpha,\frac{1}{10}\alpha^{-2}\right),\nonumber\\\label{50}
	\frac{5}{192}\frac{\min(1,2\alpha)\lambda(1-2\lambda)}{(1+\alpha)^2+\alpha^2}&\geq
		\frac{1}{96}\,\min\left(Lu\gamma^{-2}, \frac{1}{32}
		(Lu\gamma^{-2})^{-2}\right)\,\lambda. 
\end{align}

\par\medskip
2. Recall that $\Phi(r)=\int_0^r \varphi(s)\, ds$. Since $\varphi(s)\leq
	\exp(-Ls^2/8)$, we have
	\begin{eqnarray}\label{50a}
		\Phi(r) &\leq & \int_0^{\infty} \exp(-Ls^2/8)\, ds \, = \,
		\sqrt{2\pi/L}\quad\text{for any }r\geq 0.
	\end{eqnarray}
	Furthermore, by \eqref{41},
	\begin{eqnarray*}
		e^{-2}\int_0^{R_1} 1/\varphi &\leq& \int_0^{R_1}
		\exp(Ls^2/8)\, ds \,\leq\, \frac{8}{LR_1} \exp(LR_1^2/8).
	\end{eqnarray*}
	Here we have used $\int_0^x \exp(u^2/2)\,du\leq 2x^{-1}\exp(x^2/2)$ for
	$x>0$. We obtain
	\begin{eqnarray*}
		\int_0^{R_1} \frac{\Phi}{\varphi} &\leq&  \int_0^{R_1} \frac {\Phi(R_1)}{\varphi} \, \leq \,
		8\sqrt{2\pi}e^2L^{-3/2}R_1^{-1}e^{LR_1^2/8}\, = \,
		4\sqrt{\pi}e^2L^{-1}\Lambda^{-1/2}e^\Lambda,
	\end{eqnarray*}
	and thus
	\begin{eqnarray}\label{51}
		\frac{9}{2} u\gamma^{-2}\left/ \int_0^{R_1} (\Phi/\varphi)\right. \geq
		\frac{1}{18} \pi^{-1/2} e^{-2} Lu\gamma^{-2} \Lambda^{1/2} e^{-\Lambda}.
	\end{eqnarray}
\par\medskip
3. By \eqref{42} and \eqref{Lambda1}, $\eta \leq 5/6$ and hence
	\begin{eqnarray}
		\label{51a} \frac{\eta}{1+\eta} &\geq& \frac{6}{11}\eta \, = \, \frac{6}{11}
		\Lambda^{-1}.
	\end{eqnarray}
	Moreover, for $r\leq \min(2/\sqrt{L},R_1)$, 
	\begin{eqnarray*}
		r\varphi(r)/\Phi(r) &\geq & \varphi(r)\,\geq\, e^{-1} e^{-(1+\eta)Lr^2/8} \,\geq\,
		e^{-1} e^{-11/12}\, \geq e^{-2} 
	\end{eqnarray*}
	by \eqref{40} and the choice of $\epsilon$, and for $2/\sqrt{L}\leq r\leq
	R_1$, 
	\begin{eqnarray*}
		r\varphi(r)/\Phi(r) &\geq& \sqrt{L/(2\pi)} R_1 \varphi(R_1) \, \geq \, 2 \pi^{-1/2}
		\sqrt{LR_1^2/8} \, e^{-LR_1^2/8},
	\end{eqnarray*}
	where we have used that $s\varphi(s)=se^{-\beta s^2}$, $\beta\geq L/8$, is a
	decreasing function for $s\geq 2/\sqrt{L}$. Since $\Lambda=LR_1^2/8\geq 1$ by
	\eqref{Lambda1}, we obtain
	\begin{eqnarray*}
		r\varphi(r)/\Phi(r) &\geq & 2\pi^{-1/2} e^{-2} \Lambda^{1/2}
		e^{-\Lambda}\qquad\text{for all } r\in(0,R_1],
	\end{eqnarray*}
	and hence by \eqref{51a}, 
	\begin{eqnarray} \label{52}
		\frac{1}{18} \frac{\eta}{1+\eta} \inf_{s\in(0,R_1]} \frac{s\varphi(s)}{\Phi(s)}
		&\geq& \frac{1}{18} \pi^{-1/2} e^{-2} \Lambda^{-1/2}\, e^{-\Lambda}.
	\end{eqnarray}
By combining \eqref{50}, \eqref{51} and \eqref{52}, we see that the right hand
side of \eqref{46} is lower bounded by the minimum of the expression on the right hand sides of \eqref{50}, \eqref{51} and \eqref{52}. As a consequence
we see that Theorem \ref{thm:D} applies with constants $\eta$ given by
\eqref{42} and $\epsilon=4\gamma^{-1}c/(d+A)$ satisfying \eqref{44}, provided
$c\leq c_\star$, where
\begin{eqnarray*}
	c_\star &=& \frac{1}{384}\,\gamma\, \min\left(
	\lambda Lu\gamma^{-2}\,,\; \lambda (Lu\gamma^{-2})^{-2}/8\,,\;
	\,  \Lambda^{1/2} e^{-\Lambda} Lu\gamma^{-2} ,\, \Lambda^{-1/2}\,
	e^{-\Lambda}\, \right).
\end{eqnarray*}
Here we have used that $18\pi^{1/2}e^2\leq 384$. By \eqref{Lambda}, and since
$\Lambda\geq 1$ and $x^{3/2}e^{-x}<1$ for $x\geq 1$, we also have
\begin{eqnarray*}
	\frac{\lambda}{8 (Lu\gamma^{-2})^2} &\geq&
	\left(\frac{\lambda}{2Lu\gamma^{-2}}\right)^2 \frac{1}{2\lambda} \, \geq \,
	 \Lambda^{-2} \geq \Lambda^{-1/2} e^{-\Lambda}.
\end{eqnarray*}
Therefore, the second term in
the minimum defining $c_\star$ can be dropped and the value $c_\star$ is equal to the contraction rate \eqref{rate} in the assertion of Theorem \ref{thm:1}. Thus we have shown that Theorem \ref{thm:1}
is indeed a special case of Theorem \ref{thm:D}.
\end{proof}

\begin{proof}[Proof of Corollary \ref{cor:1A}]
	By \eqref{rDist}, for $(x,v),(x',v')\in\mathbb R^{2d}$,
	\begin{eqnarray}
		\left|(x,v)-(x',v')\right|^2 &\leq& (1+\gamma)^2\, \max(1,\alpha^{-2}) \,
		r((x,v),(x,v'))^2.
	\end{eqnarray}
	Moreover, if  $r:=r((x,v),(x,v'))\le \min(1,R_1)$, then by \eqref{boundf} and \eqref{fboundsTheorem1},
	\begin{equation}
		r^2\,\le\, r \, \le \, f(r)/f'_{-}(R_1) \, \le \, 2e^{2+\Lambda}
		f(r)\, \leq \,2 e^{2+\Lambda}\rho ((x,v),(x,v')),
	\end{equation}
	and if $r\ge \min(1,R_1)$, then by \eqref{ML} and since $\lambda\le 1/4$ and
	$\epsilon=4\gamma^{-1}c/(d+A)$, 
	\begin{align*}\nonumber
		r^2 \ &\le \  {16\, ((1+\alpha )^2+\alpha^2)
u}{\gamma^{-2}\epsilon^{-1} }\, \left(1+\epsilon \lyp(x,v)+\epsilon \lyp(x',v') \right)
\\ \ &\le \ {4\, ((1+\alpha )^2+\alpha^2) (d+A)
u}{  \gamma^{-1}c^{-1} }\,\rho ((x,v),(x,v'))/  f(\min(1,R_1))\\ 
\ &\le \  8e^{2+\Lambda}\,{ (1+2\alpha +2\alpha^2) (d+A)
u}{ \gamma^{-1} c^{-1} }\,\rho ((x,v),(x,v'))/ \min(1,R_1) .\nonumber
	\end{align*}
Combining the above bounds with 
Theorem \ref{thm:1} implies that
\begin{eqnarray*}
	 \mathcal W^{2}(\mu p_t,\nu p_t)^2 &\leq& C \,\mathcal W_{\rho}(\mu p_t,\nu
	 p_t) \, \leq\, Ce^{-ct} \, \mathcal W_{\rho}(\mu,\nu) \quad\text{for any }
	 t\geq 0.
\end{eqnarray*}
Uniqueness of the invariant probability measure now follows by standard arguments.
\end{proof}

\begin{proof}[Proof of Corollary \ref{cor:sdc}]
By \eqref{lambda3} and since $Lu\gamma^{-2}\le 1/30$,
\begin{equation}
\label{ratesdc1}\lambda Lu\gamma^{-2}\ \ge\ \frac{15}{8\ell}(Lu\gamma^{-2})^2.
\end{equation}
Furthermore, by \eqref{Lambda2}, \eqref{Lbeta} and
\eqref{lambda3}, $$d\ \le\ \Lambda\ \le\ 2(d+(L\mathcal R^2-\beta )/8)L\mathcal R^2/\beta\ \le\ \Lambda_1, \qquad\mbox{whence}$$
\begin{equation}
\label{ratesdc2}\Lambda^{1/2}e^{-\Lambda}Lu\gamma^{-2}\ \ge\
d^{1/2}e^{-\Lambda_1}Lu\gamma^{-2},\mbox{\quad and\quad }\Lambda^{-1/2}e^{-\Lambda}\ \ge\ \Lambda_1^{-1/2}e^{-\Lambda_1}.
\end{equation}
The first inequality in \eqref{ratesdc} now follows from Theorem \ref{thm:1}, \eqref{ratesdc1}
and \eqref{ratesdc2}. Moreover, the second inequality holds since by
\eqref{Lbeta}, $\gamma\ge \sqrt{30\, Lu}\ge \sqrt{30\, \beta u}/\mathcal R$. 
\end{proof}

\appendix
\section{Drift conditions and Lyapunov functions}\label{app:Lyapunov}

\begin{proof}[Proof of Lemma \ref{lem:L}]
By \eqref{Gen}, $\mathcal LU(x)=v\cdot\nabla U(x)$,
$\mathcal L |x|^2=2x\cdot v$,
\begin{eqnarray*}
\mathcal L\, \tfrac 12|\gamma^{-1}v|^2 &=& u\gamma^{-1}d\, -\,
\gamma^{-1}|v|^2\, -\, u\gamma^{-2}v\cdot\nabla U(x),\\
\mathcal L\, \tfrac 12|x+\gamma^{-1}v|^2 &=& u\gamma^{-1}d\,  -\, u\gamma^{-1}(x+\gamma^{-1}v)\cdot\nabla U(x),\quad\mbox{and hence}\end{eqnarray*}
\begin{eqnarray*}
\lefteqn{\mathcal L \lyp(x,v) \ =\ \tfrac 12\, \gamma\, \left( 2d\, -\,x\cdot\nabla U(x)\, -\,  u^{-1}|v|^2\, -\, \lambda u^{-1}\gamma \,  x\cdot v\right)}\\
&\le & \gamma \, \left( d+A-\lambda U(x)-\tfrac 14\lambda u^{-1}\gamma^2\, (|x|^2+2x\cdot\gamma^{-1}v+2\lambda^{-1}|\gamma^{-1}v|^2)\right)\\
&=& \gamma \, (d+A-\lambda \lyp(x,v)).
\end{eqnarray*}
\end{proof}

\begin{proof}[Proof of Lemma \ref{lem:sdc}]
Suppose that Assumption \ref{ass:2} is
satisfied, and let $A=(L\mathcal R^2-\beta )/8$. Then by
\eqref{A2prime}, \eqref{A1} and \eqref{Lbeta},
\begin{eqnarray*}
\lefteqn{x\cdot\nabla U(x) \ =\  \frac{|x|}{\mathcal R}\,\frac{\mathcal Rx}{|x|}\,
\nabla U\left( \frac{\mathcal Rx}{|x|}\right)\, +\, x\cdot\left(
\nabla U(x)-\nabla U\left(\frac{\mathcal Rx}{|x|}\right)\right)}\\
&\ge & \beta\mathcal R^{-1}|x|\, -\, L\, (\mathcal R-|x|)\, |x|
\ \ge\ \beta\mathcal R^{-2}|x|^2-(L-\beta\mathcal R^{-2})
(\mathcal R-|x|)|x|\\
&\ge &\beta\mathcal R^{-2}|x|^2\, -\, 2A
\end{eqnarray*} 
holds for $x\in\r^d$ s.t.\ $0<|x|\le\mathcal R$.
Noting that by \eqref{A0} and \eqref{A1},
$$U(x)\ \le\ U(0)\, +\, L|x|^2/2\ =\ L|x|^2/2,$$
we obtain
\begin{eqnarray*}
\lefteqn{U(x)+u^{-1}\gamma^2|x|^2/4\ \le\ (2L+u^{-1}\gamma^2)|x|^2/4}\\
&\le & ( \frac 12x\cdot\nabla U(x)+A)\, ( 1+L^{-1}u^{-1}\gamma^2/2)\, \beta^{-1}L\mathcal R^2
\end{eqnarray*}
for any $x\in\r^d$. Hence \eqref{A2} holds with $\lambda$ given
by \eqref{lambda3}. In particular, $Lu\gamma^{-2}\lambda^{-1}\ge L\mathcal
R^2/(2\beta )$, and hence by \eqref{Lambda},
\begin{equation}
\label{Lambda2a}
\Lambda\ \ge\ \frac 65\, (d+A)\, L\mathcal R^2/\beta .
\end{equation}
Now suppose first that $Lu\gamma^{-2}\le 1/8$. Then since 
$\beta\le L\mathcal R^2$, we have
\begin{eqnarray*}
\lambda &=& \frac{2\beta}{L\mathcal R^2}\,
\frac{Lu\gamma^{-2}}{1+2Lu\gamma^{-2}},\\
(1-2\lambda )^{-1} &=& \frac{1+2Lu\gamma^{-2}}{1+(2-4\beta L^{-1}\mathcal
R^{-2})Lu\gamma^{-2}}\ \le\ 1\,+\,\frac{4Lu\gamma^{-2}}{1-2Lu\gamma^{-2}} \,
\frac{\beta}{L\mathcal R^2}\\ & \le & 1\, +\, \frac{16}{3}\, \frac{\beta}{L\mathcal R^2}\, Lu\gamma^{-2},
\end{eqnarray*}
and hence by
\eqref{Lambda},
\begin{equation}
\label{Lambda2b}
\Lambda \ \le \ \frac 65 (d+A)\beta^{-1}L\mathcal R^2\, (1+2\alpha
+2\alpha^2)(1+2Lu\gamma^{-2})(1+\frac{16}{3}Lu\gamma^{-2}). 
\end{equation}
By \eqref{alphaeps}, $\alpha\le 11\, Lu\gamma^{-2}/6\le 11/48$.
Noting that $Lu\gamma^{-2}\le 1/8$,
\eqref{Lambda2} follows from \eqref{Lambda2a} and \eqref{Lambda2b}.
\end{proof}


\bibliographystyle{abbrvnat}
\bibliography{referencesfinal} 

\end{document}